\documentclass[11pt]{amsart}
\usepackage{latexsym,amssymb,amsfonts,amsmath,url,fullpage,cite}
\usepackage{color}
\usepackage[all]{xy}
\usepackage{verbatim}
\usepackage{makecell}
\usepackage{caption}
\usepackage{gensymb}
\newtheorem{theorem}{Theorem}[section]
\newtheorem{proposition}[theorem]{Proposition}
\newtheorem{lemma}[theorem]{Lemma}

\newtheorem{corollary}[theorem]{Corollary}
\newtheorem*{maintheorem}{Theorem \ref{thm:main}}
\newtheorem*{lintheorem}{Theorem \ref{thm:max_k_odd_lin}}
\newtheorem*{circtheorem}{Theorem \ref{thm:max_k_odd_circ}}
\theoremstyle{definition}
\newtheorem{definition}[theorem]{Definition}
\newtheorem{example}[theorem]{Example}
\theoremstyle{remark}
\newtheorem{remark}[theorem]{Remark}
\numberwithin{equation}{section}

\def\L{{B}^{-}_{n,k}}
\def\B{{B}}
\def\C{{B}^{\circ}_{n,k}}

\usepackage{graphicx}

\title[Chained permutations and alternating sign matrices]{Chained permutations and alternating sign matrices - inspired by three-person chess}
\author[D.\ Heuer, C.\ Morrow, B.\ Noteboom, S.\ Solhjem, J.\ Striker, C.\ Vorland]{Dylan Heuer, Chelsey Morrow, Ben Noteboom, Sara Solhjem,\\ Jessica Striker, Corey Vorland}

\begin{document}

\begin{abstract} We define and enumerate two new two--parameter permutation families, namely, placements of a maximum number of non-attacking rooks on $k$ chained-together $n\times n$ chessboards, in either a circular or linear configuration. The linear case with $k=1$ corresponds to standard permutations of $n$, and the circular case with $n=4$ and $k=6$ corresponds to a three-person chessboard. We give bijections of these rook placements to matrix form, one-line notation, and matchings on certain graphs. Finally, we define chained linear and circular alternating sign matrices, enumerate them for certain values of $n$ and $k$, and give bijections to analogues of monotone triangles, square ice configurations, and fully-packed loop configurations.
\end{abstract}

\maketitle

\tableofcontents

\section{Introduction}
A typical enumeration problem given in an introductory combinatorics course is the following: How many ways are there to place $m$ non-attacking rooks on an $n\times n$ chessboard? The solution is to first choose which $m$ rows the rooks  occupy, in $\binom{n}{m}$ ways, then the falling factorial $(n)_m := n(n-1)(n-2)\cdots (n-m+1)$ counts the number of ways to place the $m$ rooks on those $m$ rows. So there are $\binom{n}{m}\left(n\right)_m$ such rook placements. In the special case of placing the maximum number $n$ of rooks on the $n\times n$ board, this reduces to $n!$.

One natural extension of this question is to change the rules for how the piece moves. For example, one may want to count non-attacking queen placements rather than rook placements; see~\cite{qQueens1,chess_website,chess_book}. A different extension of the question is to change the chessboard. The beautiful theory of \emph{rook polynomials}, studied by Goldman, Joichi, and White in \cite{RookTheory1,RookTheory2,RookTheory3,RookTheory4,RookTheory5}, discusses the generating function of the number of rook placements on any sub-board of the $n\times n$ board and shows when the generating function of two boards is equivalent.

This paper generalizes the theory of rook placements by considering a different kind of board, namely, a board created by chaining together multiple $n\times n$ chessboards in a particular way that we describe in Definition~\ref{def:lincirc}. 

This work was inspired by the board game \emph{three-person chess}.
Though the game had been gathering dust in the fifth author's closet and the directions for game play had been lost, the board still inspired the following combinatorial question: \textbf{How many ways are there to place $m$ non-attacking rooks on the three-person chessboard of Figure~\ref{fig:3personchessboard}?} 

\begin{figure}[htbp]
\begin{center}
\includegraphics[scale=.6]{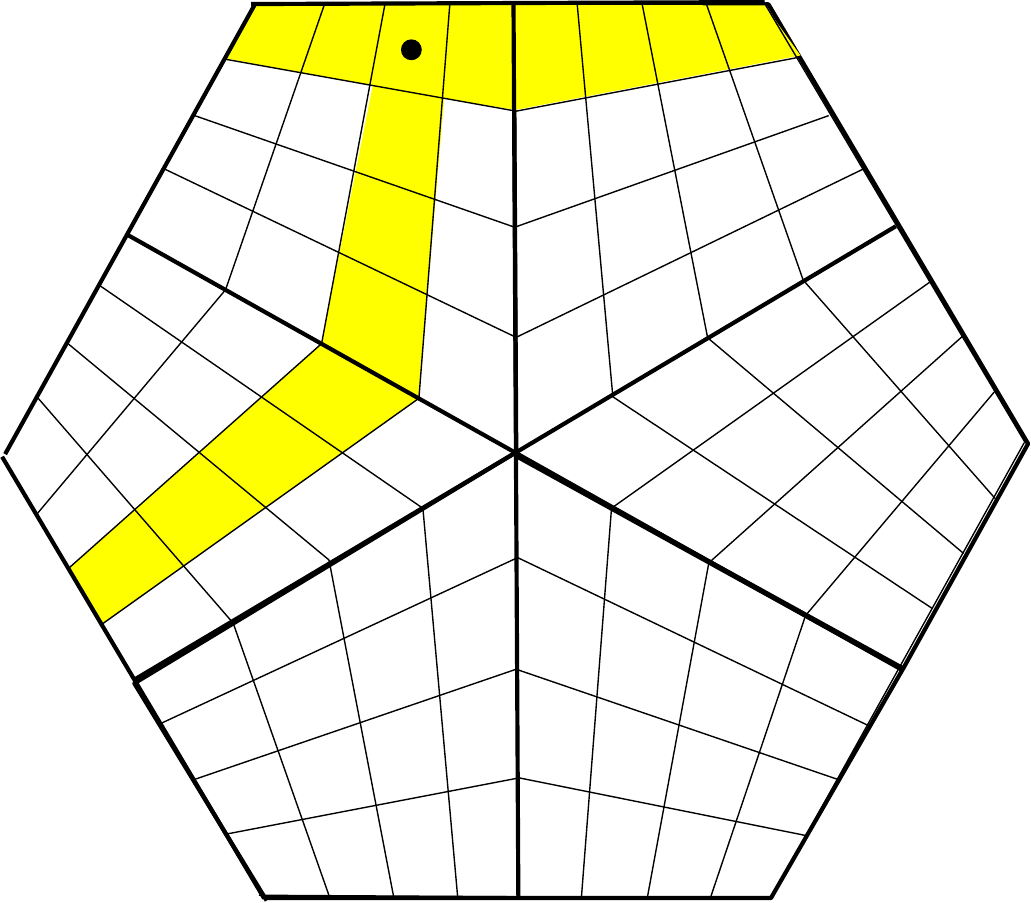}
\end{center}
\caption{A three-person chessboard; the dot represents a rook and the highlighted cells are the cells the rook is attacking. See Figures~\ref{fig:k6n4circle} and \ref{fig:k6n4circ} to see how this board transforms to $B^{\circ}_{4,6}$.}
\label{fig:3personchessboard}
\end{figure}

In this paper, we answer this question and generalize this result to a two-parameter family, namely, \emph{maximum rook placements on $k$ chained-together $n\times n$ boards} in either a \emph{linear} or \emph{circular} configuration. We highlight below our main results.

\smallskip
Our first main theorem, stated below, gives a formula for the number of non-attacking rook placements of $m$ rooks in either of these families for any values of $n$ and $k$. Let $\L$ denote the {\it linear configuration} of $k$ chained $n \times n$ chessboards and $\C$ the {\it circular configuration}; see Definition~\ref{def:lincirc}. Also, see Definition~\ref{def:comp} for the definition of $\mathfrak{C}_{m}(B)$.

\begin{maintheorem}
The number of ways to place $m$ non-attacking rooks on board $B\in \{\L, \C\}$ is \[ \sum_{(a_1, \dots, a_k) \in \mathfrak{C}_{m}(B)} \prod_{i=1}^k \binom{n-a_{i-1}}{a_i}(n)_{a_i} \] where $a_0$ is defined as follows:
\[ a_0 = \begin{cases} 
      0 & \textrm{ if $B= \L$ } \\
      a_k & \textrm{ if $B= \C$. } \\
   \end{cases} \]
\end{maintheorem}

We use this theorem to determine exact counts of placements of the \emph{maximum} number of non-attacking rooks on each board. 

\begin{lintheorem}
The number of maximum rook placements on $\L$ is given by:
\begin{itemize}
\item {Case $k$ even:} $\mbox{ }(n!)^{\frac{k}{2}} \displaystyle\sum_{0\leq j_1\leq\ldots\leq j_{\frac{k}{2}}\leq n} \mbox{ }\displaystyle\prod_{\ell=1}^{\frac{k}{2}} \binom{n-j_{\ell-1}}{n-j_{\ell}}\binom{n}{j_{\ell}}$,
\item {Case $k$ odd:} $\mbox{ }\left(n!\right)^{\frac{k+1}{2}}$.
\end{itemize}
\end{lintheorem}

\begin{circtheorem}
The number of maximum rook placements on $\C$ is given by:
\begin{itemize}
\item {Case $k$ even:}
$\mbox{ }
\left(n!\right)^{\frac{k}{2}}\displaystyle\sum_{j=0}^n \binom{n}{j}^{\frac{k}{2}}$,
\item {Case $k$ odd, $n$ even:} $\mbox{ }\left(\left(n\right)_{\frac{n}{2}}\right)^k$,
\item {Case $k$ odd, $n$ odd:} $
\mbox{ }k\lceil\frac{n}{2}\rceil\left(\left(n\right)_{\lceil\frac{n}{2}\rceil}\right)^{\lfloor\frac{k}{2}\rfloor}\left(\left(n\right)_{\lfloor\frac{n}{2}\rfloor}\right)^{\lceil\frac{k}{2}\rceil}$.
\end{itemize}
\end{circtheorem}

We then shift from discussing rook placements to the study of \emph{chained permutations}, which are equivalent to maximum rook placements on these boards. In \textbf{Theorems~\ref{prop:oneline} and \ref{prop:matching}}, we transform chained permutations into forms analogous to the one-line notation and perfect matching form of standard permutations.

Finally, we define \emph{chained alternating sign matrices} (Definition~\ref{def:chainedASM}). In \textbf{Proposition~\ref{prop:asm_n_k1} through Corollary~\ref{cor:qtasm}} we enumerate them for special values of $n$ and $k$; in \textbf{Theorems~\ref{prop:mt}, \ref{prop:sqice}, and \ref{prop:fpl}}, we transform them into forms analogous to monotone triangles, square ice configurations, and fully-packed loop configurations.

\smallskip
\textbf{Our outline is as follows.} In Section~\ref{sec:enum}, we define the boards $\L$ and $\C$ and prove Theorems~\ref{thm:main}, \ref{thm:max_k_odd_lin}, and \ref{thm:max_k_odd_circ} which enumerate non-attacking rook placements on these boards. In Section 3, we transform the maximum rook placements to chained permutations and prove Theorems~\ref{prop:oneline} and \ref{prop:matching} which give further bijections. In Section 4, we define chained alternating sign matrices, enumerate them in special cases, and prove the further bijections of Theorems~\ref{prop:mt}, \ref{prop:sqice}, and \ref{prop:fpl}.

\section{Enumeration of non-attacking rook placements on chained chessboards}
\label{sec:enum}
\subsection{Definitions and general enumeration result}
We begin by defining the boards and rook placements we will be discussing throughout this paper.

\begin{definition}
\label{def:lincirc}
Let $\L$ be a $k$-tuple $\{B^{(1)},\ldots,B^{(k)}\}$ of $n\times n$ chessboards. We say two rooks are \emph{attacking on} $\L$ if they are in the same row or column on the same board or if one is in the $j$th row of $B^{(i-1)}$ and the other is in the $j$th column of $B^{(i)}$, for some $1\leq j\leq n$ and $2\leq i\leq k$. We call $\L$ the {\it linear configuration} of $k$ chained $n\times n$ chessboards; see Figure~\ref{fig:k3n5lin}.

Let $\C$ be a $k$-tuple $\{B^{(1)},\ldots,B^{(k)}\}$ of $n\times n$ chessboards. We say two rooks are \emph{attacking on} $\C$ if they are in the same row or column on the same board or if one is in the $j$th row of $B^{(i-1)}$ and the other is in the $j$th column of $B^{(i)}$, for some $1\leq j\leq n$ and $1\leq i\leq k$, where we consider $B^{(0)}\equiv B^{(k)}$.
We call $\C$ the {\it circular configuration} of $k$ chained $n \times n$ chessboards; see Figure~\ref{fig:k6n4circ}.

A collection of rooks is \emph{non-attacking} if no pair is attacking.
\end{definition}

\begin{figure}[htbp]
\begin{center}
\includegraphics[scale=.5]{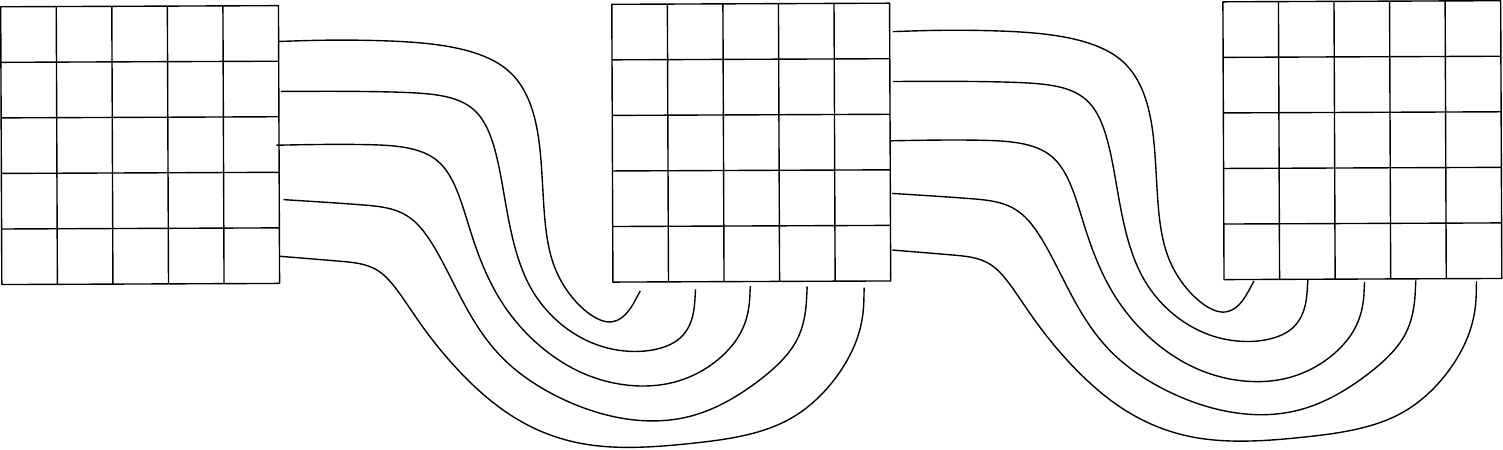}
\end{center}
\caption{The board $B^{-}_{5,3}$ from Definition~\ref{def:lincirc}, drawn with lines connecting each row of $B^{(i-1)}$ with its attacking column of $B^{(i)}$.}
\label{fig:k3n5lin}
\end{figure}

\begin{figure}[htbp]
\begin{center}
\includegraphics[scale=.5]{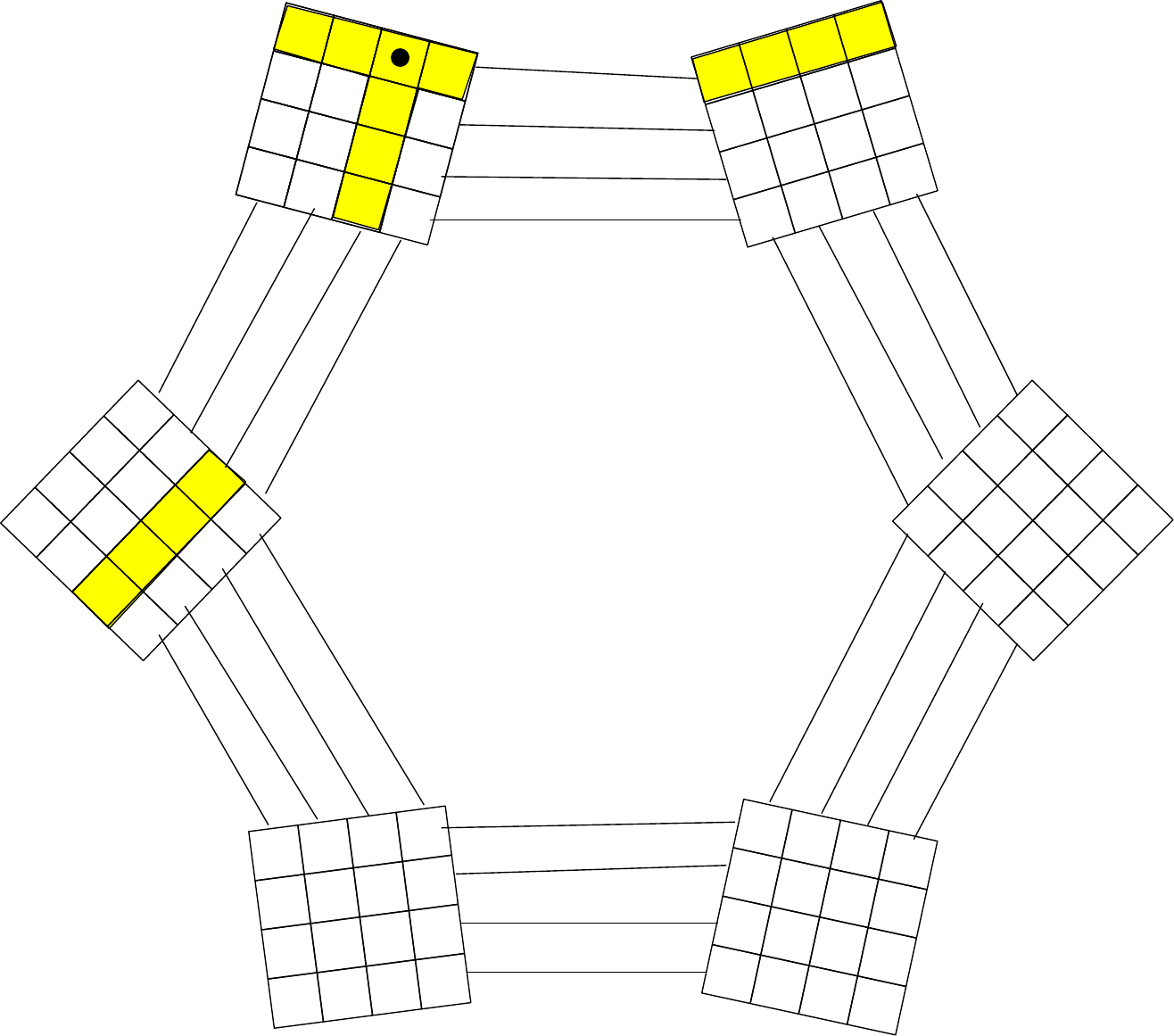}
\end{center}
\caption{The three-person chessboard from Figure~\ref{fig:3personchessboard}, expanded as a transitional step toward drawing it in the standard way of Figure~\ref{fig:k6n4circ}.}
\label{fig:k6n4circle}
\end{figure}

\begin{figure}[htbp]
\begin{center}
\includegraphics[scale=.9]{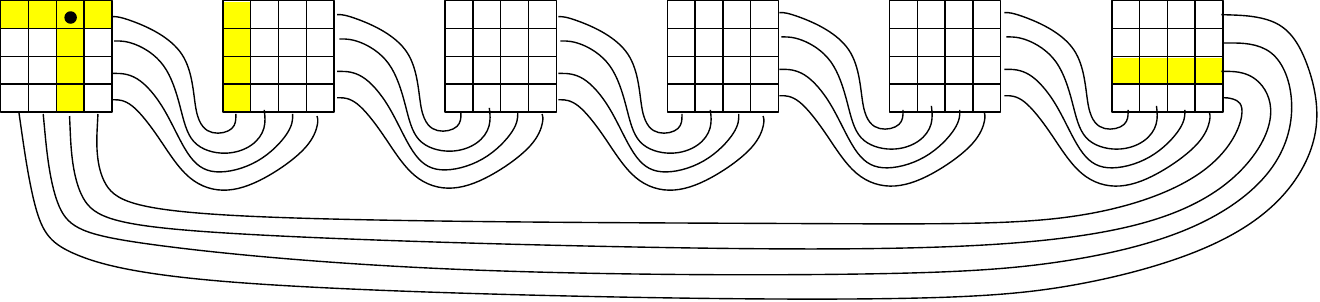}
\end{center}
\caption{The board $B^{\circ}_{4,6}$ from Definition~\ref{def:lincirc}, drawn with lines connecting each row of $B^{(i-1)}$ with its attacking column of $B^{(i)}$.}
\label{fig:k6n4circ}
\end{figure}

We now state and prove our first main result, Theorem~\ref{thm:main}.
We begin by considering the following natural questions:
\begin{enumerate}
\item What is the maximum number of non-attacking rooks we may place on $\L$ or $\C$?
\item Given a fixed number of rooks $m$, in how many different ways may we place those $m$ rooks on $\L$ or $\C$ so that they are all non-attacking?
\end{enumerate}
We  answer (1) in Lemmas~\ref{prop:maxodd} and~\ref{prop:maxeven} and (2) in Theorems~\ref{thm:main},  \ref{thm:max_k_odd_lin}, and \ref{thm:max_k_odd_circ}. First, we establish some terminology used throughout this paper.

\begin{definition}
\label{def:comp}
To each placement of non-attacking rooks on $\L$ or $\C$, associate a composition $(a_1,a_2,\ldots,a_k)$ where $a_i$ equals the number of rooks placed on $B^{(i)}$.
Define $\mathfrak{C}^{-}_{m,n,k}$ as the set of all such compositions that arise from a  placement of $m$ non-attacking rooks on $\L$. 
Define $\mathfrak{C}^{\circ}_{m,n,k}$ similarly using $\C$ instead of $\L$. 

Let $\mathfrak{C}_{m,n,k}$ denote either of $\mathfrak{C}^{-}_{m,n,k}$ or $\mathfrak{C}^{\circ}_{m,n,k}$, and let $\mathfrak{C}_m(B)$ be the set of compositions corresponding to  placements of $m$ non-attacking rooks on board $B\in\{\L,\C\}$.

See Figures~\ref{fig:k3n5circrooks} and \ref{fig:k3n5circ141} for examples.
\end{definition}

\begin{figure}[htbp]
\begin{center}
\includegraphics[scale=.5]{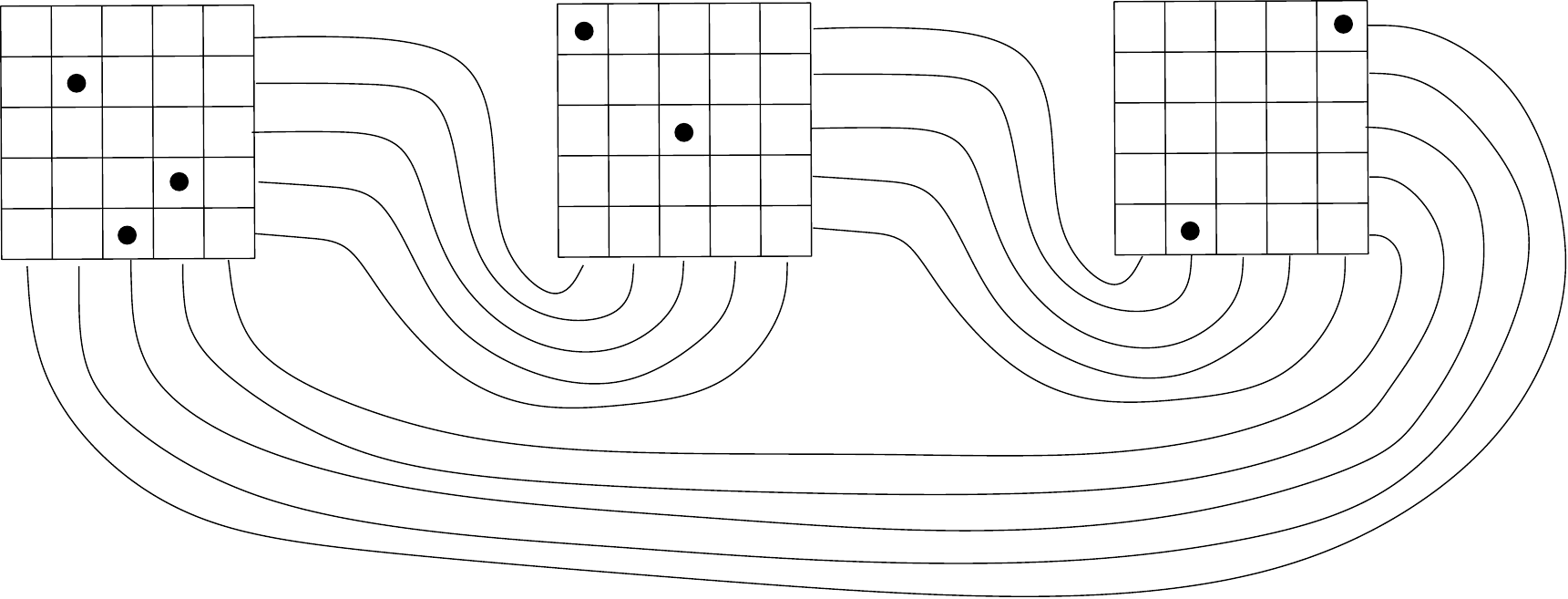}
\end{center}
\caption{A maximum non-attacking rook placement on $B^{\circ}_{5,3}$ with composition $(3,2,2)$ in $\mathfrak{C}^{\circ}_{7,3,5}$.}
\label{fig:k3n5circrooks}
\end{figure}

\begin{figure}[htbp]
\begin{center}
\includegraphics[scale=.5]{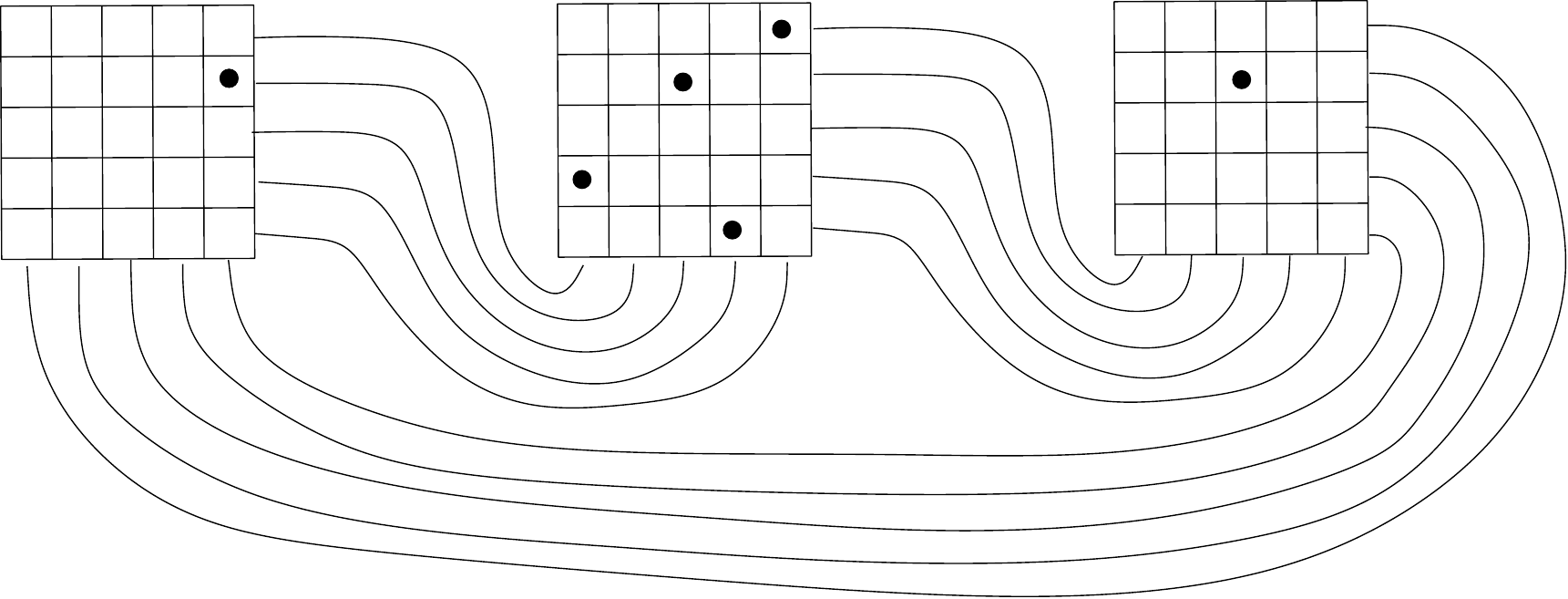}
\end{center}
\caption{A non-attacking rook placement on $B^{\circ}_{5,3}$ with composition (1,4,1) in $\mathfrak{C}^{\circ}_{6,3,5}$; note this arrangement is a maximal rook placement, but not a maximum rook placement.}
\label{fig:k3n5circ141}
\end{figure}

\begin{lemma} 
\label{lem:compositions}
$(a_1,a_2,\ldots,a_k)\in \mathfrak{C}_{m,n,k}$ if and only if 
$a_{i-1} + a_{i} \leq n$ for all $1\leq i\leq k$, where in the linear case $\mathfrak{C}^{-}_{m,n,k}$ we set $a_0=0$ and in the circular case $\mathfrak{C}^{\circ}_{m,n,k}$ we set $a_0=a_k$.
\end{lemma}

\begin{proof}
Suppose $(a_1,a_2,\ldots,a_k)\in \mathfrak{C}^{-}_{m,n,k}$. Then there exists a non-attacking rook placement on $\L$ with composition $(a_1,\ldots,a_k)$. A rook in row $j$ of $B^{(i-1)}$ would be attacking with a rook in column $j$ of $B^{(i)}$, $2\leq i\leq k$ so there may be at most one rook in each row/column pair. Thus, $a_{i-1} + a_{i} \leq n$ for $2\leq i\leq k$, and certainly $a_1\leq n$.

Suppose $(a_1,a_2,\ldots,a_k)\in \mathfrak{C}^{\circ}_{m,n,k}$. Then there exists a non-attacking rook placement on $\C$ with composition $(a_1,\ldots,a_k)$. A rook in row $j$ of $B^{(i-1)}$ would be attacking with a rook in column $j$ of $B^{(i)}$, $1\leq i\leq k$ (where $B^{(0)}\equiv B^{(k)}$) so there may be at most one rook in each row/column pair. Thus, $a_{i-1} + a_{i} \leq n$ for $2\leq i\leq k$, and $a_k+a_1\leq n$.

Suppose $(a_1,\ldots, a_k)$ satisfies $a_{i-1} + a_{i} \leq n$ for all $1\leq i\leq k$, where we set $a_0=0$. We exhibit a non-attacking rook placement in $\L$ with this composition. 
Place rooks in row $\ell_1$ column $\ell_1$ of $B^{(1)}$ for $1\leq \ell_1\leq a_{1}$. Then in $B^{(2)}$, the first $a_1$ columns cannot contain a rook, since rooks in these columns would be attacking with the rooks on $B^{(1)}$. So place rooks on $B^{(2)}$ in row $\ell_2$ column $a_1+\ell_2$ for  $1\leq \ell_2\leq a_{2}$. Continue in this way placing rooks on $B^{(i)}$ in row $\ell_i$ column $a_{i-1}+\ell_i$ for $1\leq \ell_i\leq a_{i}$. Since $(a_1,\ldots, a_k)$ satisfies $a_{i-1} + a_{i} \leq n$ for all $1\leq i\leq k$ (with $a_0=0$), no $B^{(i)}$ will run out of available columns on which to place the rooks.

Suppose $(a_1,\ldots, a_k)$ satisfies $a_{i-1} + a_{i} \leq n$ for all $1\leq i\leq k$, where we set $a_0=a_k$. We exhibit a non-attacking rook placement in $\C$ with this composition in the same way as in the linear case, except that no rook may be placed on the first $a_{1}$ rows of $B^{(k)}$, due to the rooks placed on $B^{(1)}$. So place rooks on $B^{(k)}$ in row $a_1+\ell_k$ column $a_{k-1}+\ell_k$ for $1\leq \ell_k\leq a_k$.
\end{proof}

We now present our first main result.

\begin{theorem}
\label{thm:main}
The number of ways to place $m$ non-attacking rooks on board $B\in \{\L, \C\}$ is \[ \sum_{(a_1, \dots, a_k) \in \mathfrak{C}_{m}(B)} \prod_{i=1}^k \binom{n-a_{i-1}}{a_i}(n)_{a_i} \] where $a_0$ is defined as follows:
\[ a_0 = \begin{cases} 
      0 & \textrm{ if $B= \L$ } \\
      a_k & \textrm{ if $B= \C$. } \\
   \end{cases} \]
\end{theorem}
\begin{proof}
Consider $\L$.
Fix a composition $(a_1,\ldots,a_k)\in \mathfrak{C}^{-}_{m,n,k}$ 
The number of ways to place $a_{1}$ rooks on board $B^{(1)}$ is $\binom{n}{a_1}(n)_{a_{1}}$, as discussed in the introduction.
Once we have placed $a_{1}$ rooks on $B^{(1)}$, we must then place $a_{2}$ rooks on $B^{(2)}$.  Observe that we  have $n-a_{1}$ columns in which to place $a_{2}$ rooks on $B^{(2)}$, since by Lemma~\ref{lem:compositions}, $a_1+a_2\leq n$. 
The number of ways to choose these $a_2$ columns 
from $n - a_{1}$ allowable columns is $\binom{n-a_{1}}{a_{2}}$.  Once the columns are chosen, there are $(n)_{a_2}$ ways to place the $a_2$ rooks on this board.  
Similarly, the $a_{i-1}$ rooks placed on $B^{(i-1)}$ determine the $n-a_{i-1}$ allowable columns in which the $a_i$ rooks for board $B^{(i)}$ may be placed, so there are $ \binom{n-a_{i-1}}{a_i} (n)_{a_i}$ ways to place $a_i$ rooks on $B^{(i)}$ for $2\leq i\leq k$. Thus the desired enumeration formula holds in this case.

In the case $\C$, begin by choosing the rows in which to place the $a_1$ rooks on $B^{(1)}$; this can be done in $(n)_{a_1}$ ways. Then by the same reasoning as in the linear case, there are $\binom{n-a_{i-1}}{a_i} (n)_{a_i}$ ways to place $a_i$ rooks on $B^{(i)}$ for $2\leq i\leq k$. Finally, we determine the columns in which the $a_1$ rooks on $B^{(1)}$ are to be placed. Since the rows of $B^{(k)}$ are attacking with corresponding columns of $B^{(1)}$, there are only $n-a_k$ columns on which the $a_1$ rooks may be placed, resulting in $\binom{n-a_{k}}{a_1}$ ways to choose these columns.
Thus, the total number of ways to place $m$ rooks given our chosen composition $(a_1,\ldots,a_k)$ is $\prod_{i=1}^k \binom{n-a_{i-1}}{a_i}(n)_{a_i}$. 
Summing over all compositions in $\mathfrak{C}_m(\C)$, we obtain our desired result.
\end{proof}

In the next two subsections, we investigate the maximum number of rooks we may place on $\L$ and $\C$. Once we determine this, we will use Theorem~\ref{thm:main} to find the number of non-attacking placements of these rooks. 

\begin{definition}
Let a \emph{maximum rook placement} be a non-attacking placement of the maximum number of non-attacking rooks on $\L$ or $\C$. (Note this differs from the notion of a \emph{maximal} rook placement, since there exist placements of non-attacking rooks to which no additional rooks may be added while maintaining the non-attacking property that do not achieve the maximum number of rooks for that board. The difference is illustrated in Figures~\ref{fig:k3n5circrooks} and \ref{fig:k3n5circ141}.)
\end{definition}

We start with the linear case.

\subsection{Enumeration of maximum rook placements in the linear case}

\begin{lemma}
\label{prop:maxodd}
The maximum number of non-attacking rooks that may be placed on $\L$ is $n \left\lceil\frac{k}{2}\right\rceil$.
Moreover, the compositions in $\mathfrak{C}^{-}_{n\lceil\frac{k}{2}\rceil,n,k}$ are the following:
\begin{itemize}
\item {Case $k$ even:} 
$\left(n-j_1,j_1,n-j_2,j_2,\ldots,n-j_{\frac{k}{2}},j_{\frac{k}{2}}\right)$, $0\leq j_1\leq j_2\leq \cdots\leq j_{\frac{k}{2}}\leq n$,
\item {Case $k$ odd:} 
$\left(n,0,n,\ldots,0,n\right)$.
\end{itemize}
\end{lemma}

\begin{proof}
Let $(a_1,\ldots,a_k)$ be the composition corresponding to a placement of non-attacking rooks on $\L$.

\vspace{1ex}
\textbf{Case $k$ even:} 
Observe that by Lemma~\ref{lem:compositions}, adjacent $n \times n$ boards have at most $n$ total rooks placed on them. 
So $\displaystyle\sum_{i=1}^k a_i = \sum_{\ell=1}^{\frac{k}{2}}(a_{2\ell-1}+a_{2\ell})\leq n\frac{k}{2}$.
Thus, there are at most $\frac{n k}{2}$ rooks in a non-attacking rook placement on $\L$.

To attain this maximum, we must have $a_{2\ell-1}+a_{2\ell}=n$ for all $1\leq \ell\leq \frac{k}{2}$.
So the compositions in $\mathfrak{C}^{-}_{\frac{nk}{2},n,k}$ have the form $\left(n-j_1,j_1,n-j_2,j_2,\ldots,n-j_{\frac{k}{2}},j_{\frac{k}{2}}\right)$. To satisfy the condition of Lemma~\ref{lem:compositions}, we also need $j_{\ell}+n-j_{\ell+1}\leq n$. Therefore, $j_{\ell}\leq j_{\ell+1}$ for all $1\leq\ell\leq\frac{k}{2}-1$. Thus, $0\leq j_1\leq j_2\leq \cdots\leq j_{\frac{k}{2}}\leq n$.

\vspace{1ex}
\textbf{Case $k$ odd:}
Again, by Lemma~\ref{lem:compositions}, adjacent $n \times n$ boards have at most $n$ total rooks placed on them.
So $\displaystyle\sum_{i=1}^k a_i = a_1 + \sum_{\ell=1}^{\frac{k-1}{2}}(a_{2\ell}+a_{2\ell+1})\leq  n+n\frac{k-1}{2} = n \left\lceil\frac{k}{2}\right\rceil$.
Thus, there are at most $n \left\lceil\frac{k}{2}\right\rceil$ rooks in a non-attacking rook placement on $\L$.

 We construct a non-attacking rook placement on $\L$ with exactly $n \left\lceil\frac{k}{2}\right\rceil$ rooks as follows. Place $n$ non-attacking rooks on each of  $B^{(2\ell-1)}$ for $1\leq\ell\leq \frac{k+1}{2}$. Such a placement has composition $\left(n,0,n,\ldots,0,n\right)$. 

We now show that this is the only way to place $\frac{n(k+1)}{2}$ non-attacking rooks on the board. 
When $k=1$, this is clear.  
Now consider a non-attacking rook placement on $\L$ with $k>1$.
Suppose  $a_{2\ell}\neq 0$.  
Then we have $n-a_{2\ell}$ rows in which to place a rook on $B^{(2\ell-1)}$ and $n-a_{2\ell}$ columns in which to place a rook on $B^{(2\ell+1)}$.
So $a_{2\ell-1}+a_{2\ell}+a_{2\ell+1}\leq 2(n-a_{2\ell})+a_{2\ell}=2n-a_{2\ell}$.  
By the even case, we know that
$\displaystyle\sum_{j=1}^{2\ell-2}a_j\leq \frac{n (2\ell-2)}{2}$, and similarly, $\displaystyle\sum_{j=2\ell+2}^{k}a_j\leq\frac{n(k-(2\ell+2)+1)}{2}=\frac{n(k-2\ell-1)}{2}$.  So there may be at most $\displaystyle\frac{n (2\ell-2)}{2} + 2n-a_{2\ell} + \displaystyle\frac{n(k-2\ell-1)}{2}  =   \displaystyle\frac{n(k+1)}{2} -a_{2\ell} < \displaystyle\frac{n(k+1)}{2}$
non-attacking rooks placed on $\L$. 
Thus, if $a_{2\ell}\neq 0$, we do not obtain a maximum rook placement.  Therefore, the only way we may obtain a maximum rook placement is by placing $n$ rooks on each of $B^{(2\ell-1)}$ for $1\leq \ell\leq \frac{k+1}{2}$ and zero rooks on the remaining boards. 
\end{proof}

We now state and prove our second main result, which enumerates the  number of maximum rook placements on $\L$; see Figures~\ref{fig:k3n5linrooks} and \ref{fig:k4n5linrooks} for examples.

\begin{figure}[htbp]
\begin{center}
\includegraphics[scale=.6]{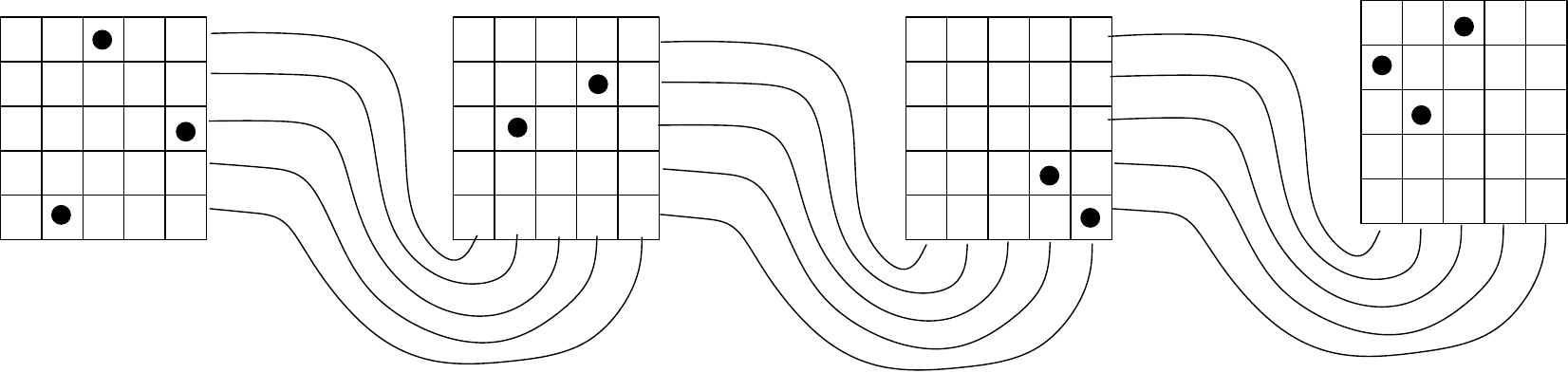}
\end{center}
\caption{A maximum rook placement on a $k$ even linear board.}
\label{fig:k4n5linrooks}
\end{figure}

\begin{figure}[htbp]
\begin{center}
\includegraphics[scale=.5]{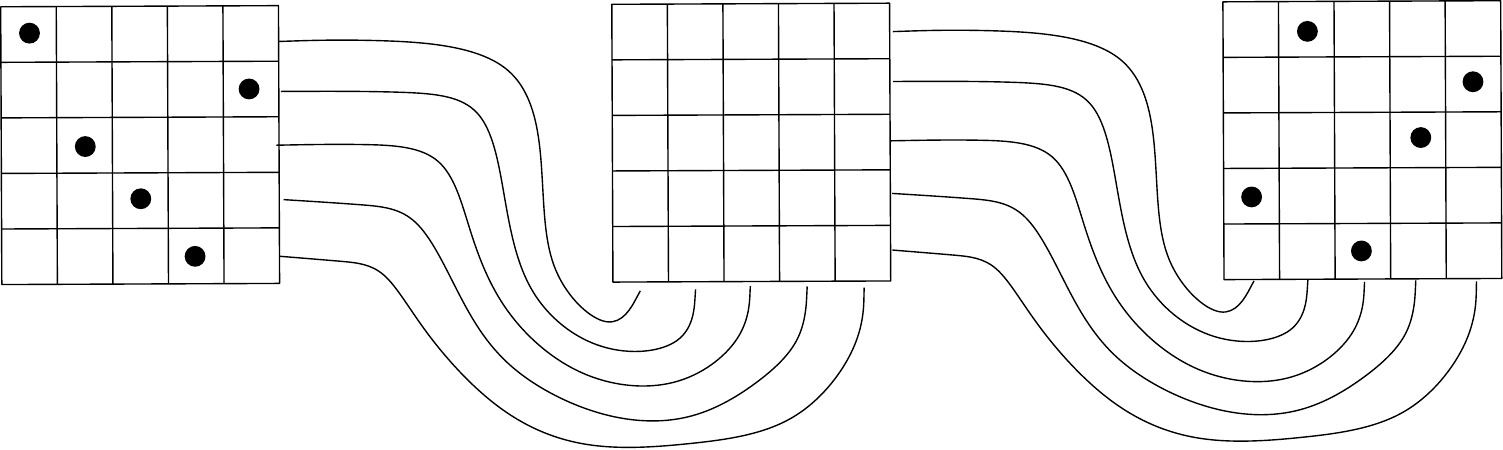}
\end{center}
\caption{A maximum rook placement on a $k$ odd  linear board.}
\label{fig:k3n5linrooks}
\end{figure}

\begin{theorem}
\label{thm:max_k_odd_lin}
The number of maximum rook placements on $\L$ is given by:
\begin{itemize}
\item {Case $k$ even:} $\mbox{ }(n!)^{\frac{k}{2}} \displaystyle\sum_{0=j_0\leq j_1\leq\ldots\leq j_{\frac{k}{2}}\leq n} \mbox{ }\displaystyle\prod_{\ell=1}^{\frac{k}{2}} \binom{n-j_{\ell-1}}{n-j_{\ell}}\binom{n}{j_{\ell}}$,
\item {Case $k$ odd:} $\mbox{ }\left(n!\right)^{\frac{k+1}{2}}$.
\end{itemize}
\end{theorem}
\begin{proof}
{Case $k$ even:}
Recall from Lemma~\ref{prop:maxodd} the compositions in $\mathfrak{C}^{-}_{\frac{nk}{2},n,k}$ have the form \[\left(n-j_1,j_1,n-j_2,j_2,\ldots,n-j_{\frac{k}{2}},j_{\frac{k}{2}}\right)\] for some $0\leq j_1\leq j_2\leq \cdots\leq j_{\frac{k}{2}}\leq n$.

We apply Theorem~\ref{thm:main} to these compositions and see the number of maximum rook placements on $\L$ when $k$ is even is given by

\[\displaystyle\sum_{(a_1, \dots, a_k) \in \mathfrak{C}^{-}_{\frac{n k}{2},n,k}} \prod_{i=1}^k (n)_{a_i} \binom{n-a_{i-1}}{a_i} 
= 
\displaystyle\sum_{0\leq j_1\leq\ldots\leq j_{\frac{k}{2}}\leq n}\mbox{ }\prod_{\ell=1}^{\frac{k}{2}}(n)_{n-j_{\ell}}\binom{n-j_{\ell-1}}{n-j_{\ell}}(n)_{j_{\ell}}\binom{n-(n-j_{\ell})}{j_{\ell}}
\]
which, after some algebraic manipulation, yields the desired result.

\textbf{Case $k$ odd:}
By Lemma~\ref{prop:maxodd}, the only way to obtain a maximum rook placement is by placing $n$ rooks on the odd numbered boards.  There are $n!$ ways to place $n$ non-attacking rooks on one $n \times n$ board.  Since we are placing $n$ rooks on each odd numbered board, of which there are $\frac{k+1}{2}$, the total number of rook placements on $\L$ is $\left(n!\right)^{\frac{k+1}{2}}$. 
\end{proof}

\begin{remark}
We can rewrite the formula for the number of maximum rook placements on $\L$ when $k$ is even in terms of multinomial coefficients as follows:
\[(n!)^{\frac{k}{2}} \displaystyle\sum_{0\leq j_1\leq\ldots\leq j_{\frac{k}{2}}\leq n} \displaystyle\binom{n}{n-j_{\frac{k}{2}},j_{\frac{k}{2}}-j_{\frac{k}{2}-1},\ldots,j_2-j_1,j_1}\displaystyle\prod_{\ell=1}^{\frac{k}{2}} \binom{n}{j_{\ell}}.\]
\end{remark}

\subsection{Enumeration of maximum rook placements in the circular case}
We now investigate the circular case. We begin by determining the maximum number of rooks one may place on $\C$.  
\begin{lemma}
\label{prop:maxeven}
The maximum number of rooks that one may place on $\C$ 
is $\left\lfloor\frac{nk}{2}\right\rfloor$. 
Moreover, the compositions in $\mathfrak{C}^{\circ}_{\lfloor\frac{nk}{2}\rfloor,n,k}$ are the following:
\begin{itemize}
\item {Case $k$ even:} $(n-j,j,n-j,j,\cdots ,n-j,j)$, $0\leq j\leq n$,
\item {Case $k$ odd, $n$ even:} $\left(\frac{n}{2},\frac{n}{2},\ldots,\frac{n}{2}\right)$, 
\item {Case $k$ odd, $n$ odd:} Cyclic shifts of  $\left(\frac{n-1}{2},\frac{n+1}{2},\frac{n-1}{2},\ldots, \frac{n+1}{2},\frac{n-1}{2}\right)$.
\end{itemize}
\end{lemma}

\begin{proof}
Let $(a_1,\ldots,a_k)$ be the composition corresponding to a placement of non-attacking rooks on $\C$. Also, set $a_0=a_k$.

\vspace{1ex}
\textbf{Case $k$ even:} 
Observe that by Lemma~\ref{lem:compositions}, adjacent $n \times n$ boards have at most $n$ total rooks placed on them.  
So $\displaystyle\sum_{i=1}^k a_i = \frac{1}{2} \sum_{i=1}^{k}(a_{i-1}+a_{i})\leq \frac{1}{2} n k$.
Thus, there are at most $\frac{n k}{2}$ rooks in a non-attacking rook placement on $\C$.

To attain this maximum, we must have $a_{i-1}+a_{i}=n$ for all $1\leq i\leq k$.
So the compositions in $\mathfrak{C}^{\circ}_{\frac{nk}{2},n,k}$ have the form $\left(n-j,j,n-j,j,\ldots,n-j,j\right)$. 

\vspace{1ex}
\textbf{Case $k$ odd, $n$ even:}  
Again, by Lemma~\ref{lem:compositions}, adjacent $n \times n$ boards have at most $n$ total rooks placed on them. 
So $\displaystyle\sum_{i=1}^k a_i = \frac{1}{2} \sum_{i=1}^{k}(a_{i-1}+a_{i})\leq \frac{1}{2} n k$.
Thus, there are at most $\frac{n k}{2}$ rooks in a non-attacking rook placement on $\C$.

To attain this maximum, we must have $a_{i-1}+a_{i}=n$ for all $1\leq i\leq k$.
So the compositions in $\mathfrak{C}^{\circ}_{\frac{nk}{2},n,k}$ have the form $\left(n-j,j,n-j,j,\ldots, j, n-j \right)$. But then $a_1+a_k=2n-2j$, which must also equal $n$. So $n=2j$, that is, our composition is $\left(\frac{n}{2},\frac{n}{2},\ldots,\frac{n}{2}\right)$.

\vspace{1ex}
\textbf{Case $k$ odd, $n$ odd:} 
Again, by Lemma~\ref{lem:compositions}, adjacent $n \times n$ boards have at most $n$ total rooks placed on them. 
So $\displaystyle\sum_{i=1}^k a_i = \frac{1}{2} \sum_{i=1}^{k}(a_{i-1}+a_{i})\leq \frac{1}{2} n k$.
Thus, there are at most $\left\lfloor\frac{n k}{2}\right\rfloor$ rooks in a non-attacking rook placement on $\C$.

$\left(\frac{n-1}{2},\frac{n+1}{2}, \frac{n-1}{2},\cdots,\frac{n+1}{2},\frac{n-1}{2}\right)$ is a composition of $\left\lfloor\frac{nk}{2}\right\rfloor$ that satisfies the condition of Lemma~\ref{lem:compositions}.
We claim cyclic shifts of $\left(\frac{n-1}{2},\frac{n+1}{2}, \frac{n-1}{2},\cdots,\frac{n+1}{2},\frac{n-1}{2}\right)$ are the only compositions that result in a maximum rook placement.

Suppose some $a_i < \frac{n-1}{2}$. Consider the linear board  $\B^{-}_{n,k-1}$
obtained by removing $B^{(i)}$ from $\C$. By Lemma~\ref{prop:maxodd}, there can be at most $\frac{n(k-1)}{2}$ rooks on  $\B^{-}_{n,k-1}$. So there can be at most $\frac{n(k-1)}{2} + a_i$ rooks on $\C$. However, $\frac{n(k-1)}{2} + a_i < \frac{n(k-1)}{2} + \frac{n-1}{2}
=\left\lfloor\frac{nk}{2}\right\rfloor$, and so this configuration cannot have a maximum rook placement. 

Suppose some $a_i > \frac{n+1}{2}$. Since adjacent boards can have a total of at most $n$ rooks between them, $a_{i-1} \leq n-a_i < \frac{n-1}{2}$. We may then apply the previous case to see that this configuration cannot have a maximum rook placement. 

Therefore, a maximum rook placement on $\C$ with $k$ odd and $n$ odd must have boards with either $\frac{n-1}{2}$ rooks or $\frac{n+1}{2}$ rooks placed on them. Note that we cannot have two adjacent boards each with $\frac{n+1}{2}$ rooks on them. Therefore, because $k$ is odd, we can have at most $\frac{k-1}{2}$ boards with $\frac{n+1}{2}$ rooks. By placing $\frac{n+1}{2}$ rooks on boards such that no two of these boards are adjacent and placing $\frac{n-1}{2}$ rooks on the remaining boards, we get a maximum rook placement and thus verify the claim.
\end{proof}

We now state and prove our third main enumerative result; see Figures~\ref{fig:k6n4circrooks}, \ref{fig:k3n4circrooks}, and \ref{fig:k3n5circrooks} for examples.

\begin{figure}[htbp]
\begin{center}
\includegraphics[scale=.9]{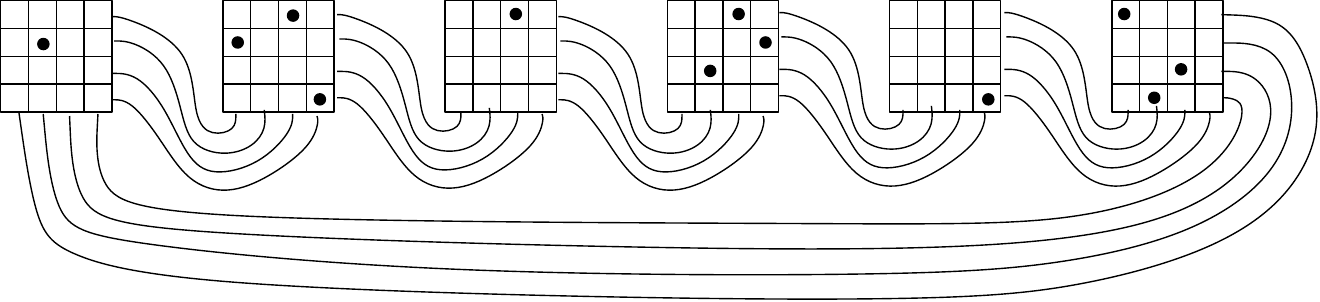}
\end{center}
\caption{A maximum rook placement on a  $k$ even circular board.}
\label{fig:k6n4circrooks}
\end{figure}

\begin{figure}[htbp]
\begin{center}
\includegraphics[scale=.6]{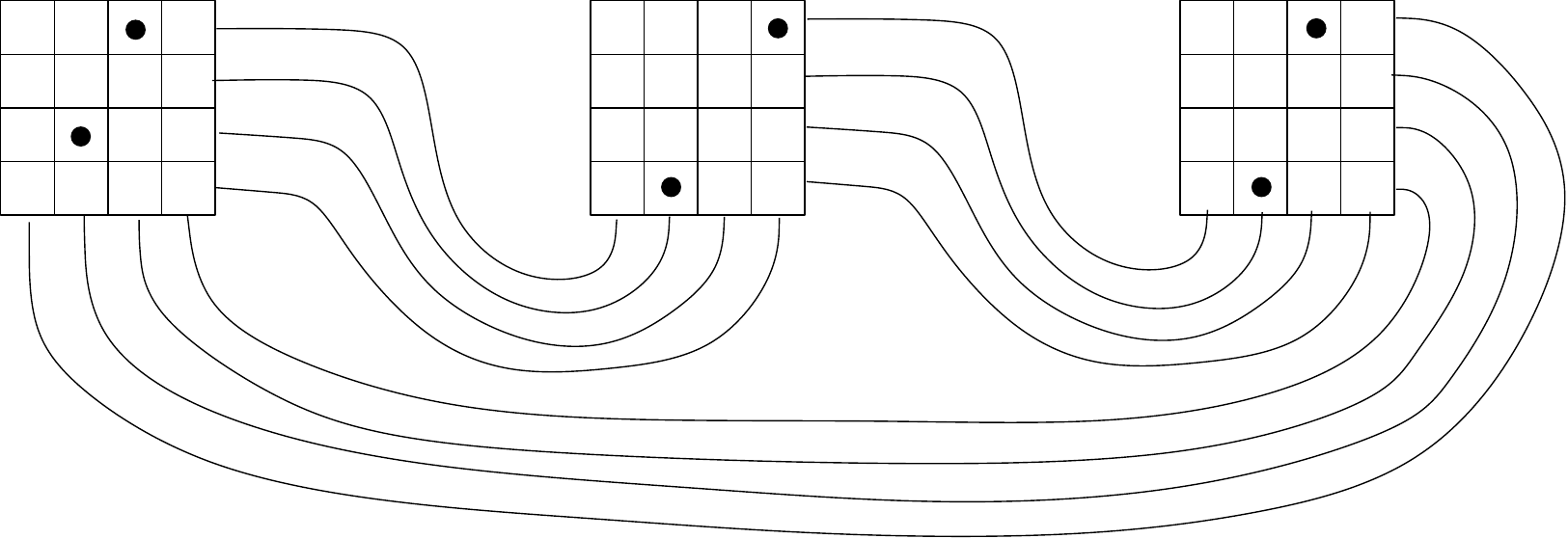}
\end{center}
\caption{A maximum rook placement on a $k$ odd and $n$ even circular board.}
\label{fig:k3n4circrooks}
\end{figure}

\begin{theorem}
\label{thm:max_k_odd_circ}
The number of maximum rook placements on $\C$ is given by:
\begin{itemize}

\item {Case $k$ even:}
$\mbox{ }
\left(n!\right)^{\frac{k}{2}}\displaystyle\sum_{j=0}^n \binom{n}{j}^{\frac{k}{2}}$,
\item {Case $k$ odd, $n$ even:} $
\mbox{ }\displaystyle \left(\left(n\right)_{\frac{n}{2}}\right)^k$,
\item {Case $k$ odd, $n$ odd:} 
$\mbox{ }k
\lceil\frac{n}{2}\rceil
\left(\left(n\right)_{\lceil\frac{n}{2}\rceil}\right)^{\lfloor\frac{k}{2}\rfloor}\left(\left(n\right)_{\lfloor\frac{n}{2}\rfloor}\right)^{\lceil\frac{k}{2}\rceil}$.
\end{itemize}
\end{theorem}

\begin{proof}
\textbf{Case $k$ even:}
Recall from Lemma~\ref{prop:maxeven} that the compositions in $\mathfrak{C}^{\circ}_{\frac{nk}{2},n,k}$ for $k$ even are $(n-j,j,n-j,j,\cdots ,n-j,j)$, $0\leq j\leq n$. 

We apply the formula of Theorem~\ref{thm:main} to these compositions, along with some algebraic manipulation, and see the number of maximum rook placements on $\C$ when $k$ is even is \[\displaystyle\sum_{j=0}^n \left(\left(n\right)_{n-j}\right)^{\frac{k}{2}}\left(\left(n\right)_j\right)^{\frac{k}{2}}=
\displaystyle\sum_{j=0}^n \left(\frac{n!}{j!}\frac{n!}{(n-j)!}\right)^{\frac{k}{2}}=\left(n!\right)^{\frac{k}{2}}\sum_{j=0}^n \binom{n}{j}^{\frac{k}{2}}.\]

\textbf{Case $k$ odd, $n$ even:} 
As discussed in Lemma~\ref{prop:maxeven}, the only composition in $\mathfrak{C}^{\circ}_{\frac{nk-1}{2},n,k}$  
is $\left(\frac{n}{2},\frac{n}{2},\frac{n}{2},\cdots ,\frac{n}{2}\right)$. Now applying the formula in Theorem~\ref{thm:main}, we find the  number of maximum rook placements is \[\sum_{(a_1, \dots, a_k) \in \mathfrak{C}_{m,n,k}} \prod_{i=1}^k (n)_{a_i} \binom{n-a_{i-1}}{a_i}= \left( (n)_{\frac{n}{2}}\binom{n-\frac{n}{2}}{\frac{n}{2}}\right)^k = \left(\left(n\right)_{\frac{n}{2}}\right)^k.\]

\textbf{Case $k$ odd, $n$ odd:} 
As discussed in Lemma~\ref{prop:maxeven}, each composition in $\mathfrak{C}^{\circ}_{\lfloor\frac{nk}{2}\rfloor,n,k}$ for $k$ odd is a cyclic shift of $(\frac{n-1}{2},\frac{n+1}{2},\frac{n-1}{2}, \frac{n+1}{2}, \ldots, \frac{n+1}{2}, \frac{n-1}{2})$. 
Now there are $\frac{k-1}{2}$ entries equal to $\frac{n+1}{2}$ in such a composition and 
$\frac{k+1}{2}$ entries equal to $\frac{n-1}{2}$. 
There are $k$ such compositions, so applying Theorem~\ref{thm:main}, we obtain $ k\lceil\frac{n}{2}\rceil\left(\left(n\right)_{\frac{n+1}{2}}\right)^{\frac{k-1}{2}}\left(\left(n\right)_{\frac{n-1}{2}}\right)^{\frac{k+1}{2}}$.
This completes the proof.
\end{proof}

\begin{remark}
We can write the sum in the $k$ even case of Theorem~\ref{thm:max_k_odd_circ} above as a generalized hypergeometric function, yielding
\[\left(n!\right)^{\frac{k}{2}}\sum_{j=0}^n \binom{n}{j}^{\frac{k}{2}}=\left(n!\right)^{\frac{k}{2}} \mbox{ } { }_{\frac{k}{2}}\mathrm{F}_{\frac{k}{2}-1}\left(-n,\ldots,-n; 1,\ldots,1; \left(-1\right)^{\frac{k}{2}+1}\right)\]
where ${ }_p \mathrm{F}_q \left(a_1,\ldots,a_p; b_1,\ldots,b_q;z\right):=\displaystyle\sum_{i=0}^{\infty}\displaystyle\frac{(a_1)^{(i)}(a_2)^{(i)}\cdots(a_p)^{(i)}}{(b_1)^{(i)}(b_2)^{(i)}\cdots(b_q)^{(i)}}\displaystyle\frac{z^i}{i!}$ and we use $(a)^{(i)}$ to denote the rising factorial $a(a+1)(a+2)\cdots (a+i-1)$. 

For $k=2$, ${ }_{\frac{k}{2}}\mathrm{F}_{\frac{k}{2}-1}\left(-n,\ldots,-n; 1,\ldots,1; \left(-1\right)^{\frac{k}{2}+1}\right)$ reduces to $2^n$, and for $k=4$ it reduces to $\binom{2n}{n}$. But for $k=6$ and greater, there is no closed form expression, so this formula is the best possible~\cite{Petkovsek}. See~\cite{wolfram} for more on these generalized hypergeometric functions.
\end{remark}

\smallskip
\begin{remark}
\label{rem:skew}
The linear board $\L$ is equivalent to a certain skew partition shape inside an $n \left\lceil\frac{k+1}{2}\right\rceil \times n \left\lceil\frac{k}{2}\right\rceil $ chessboard. For $k$ even, $\L$ is equivalent to the subboard of skew partition shape $\left(\frac{nk}{2}\right)^{2n}\left(\frac{n(k-2)}{2}\right)^n\cdots(2n)^n n^n / \left(\frac{n(k-2)}{2}\right)^n\left(\frac{n(k-4)}{2}\right)^n\cdots(2n)^n n^n$. For $k$ odd, $\L$ is equivalent to the subboard of skew partition shape $\left(\frac{n(k+1)}{2}\right)^{2n}\left(\frac{n(k-1)}{2}\right)^n\cdots(2n)^n  / \left(\frac{n(k-1)}{2}\right)^n\left(\frac{n(k-3)}{2}\right)^n\cdots(2n)^n n^n$. 
The circular board $\C$ for $k$ even and greater than $2$ is equivalent to the union of the corresponding skew partition shape for $B_{n,k-1}^{-}$ and $n^n$ in the upper left corner. 
So the enumerations of maximum rook placements in these cases could be computed, alternatively, by the theory of rook polynomials, rather than the direct combinatorial arguments given in this paper. 
\end{remark}

\section{Chained permutations}
Now that we have defined and enumerated linear and circular chained maximum rook placements, we relate these to some constructs from the theory of permutations.

\subsection{Definition}
On a regular $n\times n$ board, maximum rook placements correspond to permutations. In Section~\ref{sec:enum}, we defined and enumerated two new two-parameter families of maximum rook placements. 
In this section, we consider these as new families of permutations. We give a formal definition as follows.

\begin{definition}
\label{def:permmatrices}
Define the sets of \emph{chained linear and circular permutation matrices}, denoted $P_{n,k}^{-}$ and ${P}_{n,k}^{\circ}$,
as $k$-tuples of $n\times n$ $\{0,1\}$--matrices $\left(X^{(1)},X^{(2)},\ldots,X^{(k)}\right)$ satisfying:

\begin{enumerate}
\item  $\displaystyle\sum_{j=1}^n X^{(\ell-1)}_{i,j} + \displaystyle\sum_{j=1}^n X^{(\ell)}_{j,i}\in\{0,1\}$  for all $1\leq i\leq n, 1\leq\ell\leq k$, and 
\item the sum of all entries $\displaystyle\sum_{i,j,\ell}  X^{(\ell)}_{i,j}$ is maximum,
\end{enumerate}
where for $P_{n,k}^{\circ}$ we consider 
$X^{(0)}\equiv X^{(k)}$ and for  $P_{n,k}^{-}$ we consider 
$X^{(0)}$ to be the zero matrix. 
Let ${P}_{n,k}$ denote either  ${P}_{n,k}^{-}$ or ${P}_{n,k}^{\circ}$, depending on context.
\end{definition}

See Figure~\ref{fig:k6n4circperm} 
for an example of chained permutation matrix.

\begin{figure}[htbp]
\begin{center}
\includegraphics[scale=.95]{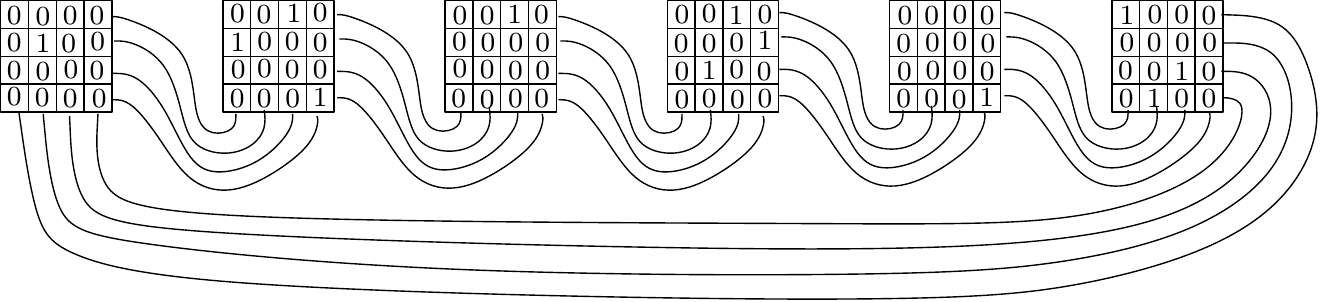}
\end{center}
\caption{The chained circular permutation matrix corresponding to the maximum rook placement of Figure~\ref{fig:k6n4circrooks}.}
\label{fig:k6n4circperm}
\end{figure}

\begin{proposition}
\label{prop:rookisperm}
$P_{n,k}$ is in bijection with the set of chained maximum rook placements on $B_{n,k}$.
\end{proposition}

\begin{proof}
As in the case of standard permutations, let a rook be represented by a one and an unoccupied space on the board as a zero. The claim then follows directly.
\end{proof}

We make the following enumerative observations, which are clear from the definitions.
\begin{remark}
$P_{n,1}^{-}$ corresponds to standard permutations of $n$. $P_{n,4}^{\circ}$ is equivalent to permutations of $2n$, since the four matrices can be combined to make a $2n\times 2n$ permutation matrix.
\end{remark}

\subsection{Chained permutation bijections}
In this section, we transform chained permutations into forms analogous to the one-line notation and perfect matching form
of standard permutations.

\begin{definition}
\label{def:oneline}
Let the \emph{one-line notation} of a chained permutation be constructed as 

\[ p^{(1)}_1 p^{(1)}_2 \ldots p^{(1)}_n - p^{(2)}_1 p^{(2)}_2 \ldots p^{(2)}_n - \cdots - p^{(k)}_1 p^{(k)}_2 \ldots p^{(k)}_n \]

\noindent
where $p^{(\ell)}_i=j$ if $X^{(\ell)}_{i,j}=1$ and $p^{(\ell)}_i=0$ if $X^{(\ell)}_{i,j}=0$ for all $1\leq j\leq n$. That is, $p^{(\ell)}_i$ records the column of the unique 1 in row $i$ of the $\ell$th matrix if there is a 1 in that row, or zero if the $i$th row is all zeros. Note, in the circular case, we append a dash to the end to indicate $ p^{(k)}_1 p^{(k)}_2 \ldots p^{(k)}_n$ is chained to $p^{(1)}_1 p^{(1)}_2 \ldots p^{(1)}_n$.
\end{definition}

\begin{example}
The one-line notation corresponding to the chained permutation of Figure~\ref{fig:k6n4circperm} is:
\[0200-3104-3000-3420-0004-1032-.
\]

The one-line notation corresponding to the maximum rook placement of Figure~\ref{fig:k4n5linrooks} is:
\[30502-04200-00045-31200.\]
\end{example}

\begin{example}
The chained circular permutations in $P^{\circ}_{2,2}$ in one-line notation are as follows:
\[12-00-, \ 21-00-, \ 00-12-, \ 00-21-, \ 10-02-, \ 01-01-, \ 20-20-, \ 02-10-.\]
See also Figure~\ref{fig:k2n2circperm}.
\end{example}

\begin{figure}[htbp]
\begin{center}
\includegraphics[scale=.5]{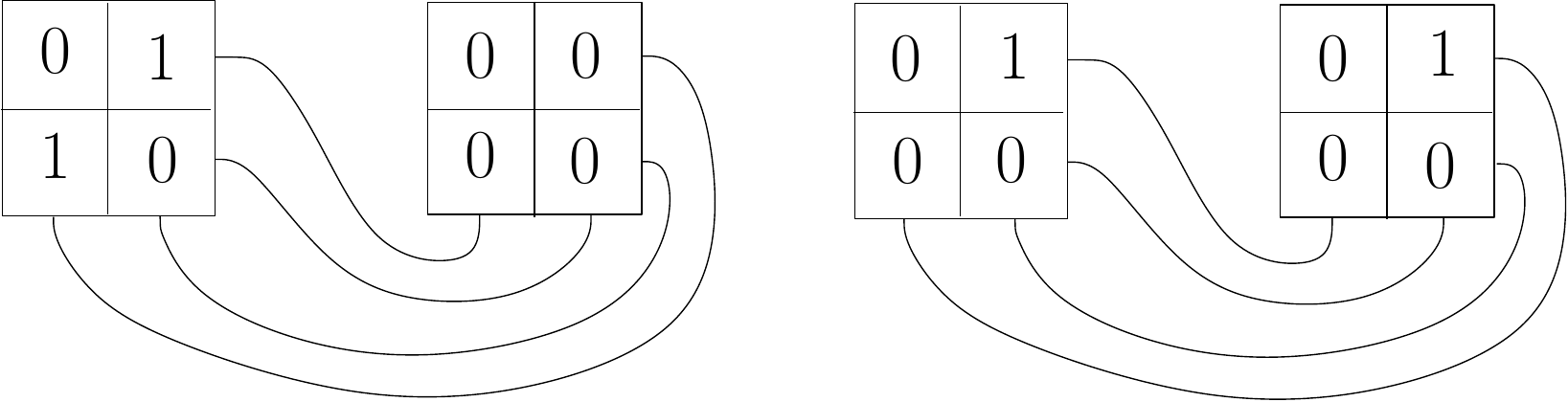}
\end{center}
\caption{The chained circular permutation matrices corresponding to the one-line notation chained permutations $21-00-$ and $20-20-$.}
\label{fig:k2n2circperm}
\end{figure}

The one-line notation of a chained permutation may be described without reference to the matrix form as follows.

\begin{theorem}
\label{prop:oneline}
\[ p^{(1)}_1 p^{(1)}_2 \ldots p^{(1)}_n - p^{(2)}_1 p^{(2)}_2 \ldots p^{(2)}_n - \cdots - p^{(k)}_1 p^{(k)}_2 \ldots p^{(k)}_n \]
is the one-line notation of a chained permutation in $P_{n,k}$ if and only if it satisfies the following for all $1\leq \ell\leq k$:
\begin{enumerate}
\item each $p^{(\ell)}_i$ is an integer with $0\leq p^{(\ell)}_i\leq n$, 
\item if $p^{(\ell)}_i=p^{(\ell)}_j$ then $p^{(\ell)}_i=0$,
\item the number of nonzero entries equals $\frac{n(k+1)}{2}$ for $k$ odd linear and $\lfloor\frac{nk}{2}\rfloor$ in all other cases, and
\item if $p^{(\ell-1)}_i\neq 0$ then 
$p^{(\ell)}_{j}\neq i$ for all $1\leq j\leq n$, 
\end{enumerate}
where for $P^{\circ}_{n,k}$, we consider $p^{(0)}_i= p^{(k)}_i$  and for $P^{-}_{n,k}$ we consider $p^{(0)}_i=0$ for all $1\leq i\leq n$.
\end{theorem}
\begin{proof} 
(1) and (2) are clear from the construction in Definition~\ref{def:oneline}.
(3) follows from the second condition in Definition~\ref{def:permmatrices} and Lemmas~\ref{prop:maxodd} and~\ref{prop:maxeven}. 
(4) is equivalent to the first condition of Definition~\ref{def:permmatrices} which determines the chaining of the matrices. 

Given $p^{(1)}_1 p^{(1)}_2 \ldots p^{(1)}_n - p^{(2)}_1 p^{(2)}_2 \ldots p^{(2)}_n - \cdots - p^{(k)}_1 p^{(k)}_2 \ldots p^{(k)}_n$ satisfying the above conditions, reconstruct the chained permutation matrix by setting $X^{(\ell)}_{i,j}=1$ if $p^{(\ell)}_i=j$ and  $0$ in all other entries. Conditions (3) and (4) above guarantee that the matrix satisfies the conditions of Definition~\ref{def:permmatrices}.
\end{proof}

We now define a graph whose matchings we show are in bijection with chained permutations.
\begin{definition}
Construct the graph $G^{-}_{n,k}$ as follows. Begin with a grid with rows $0$ through $k$ of $n$ vertices each. Between the vertices of rows $i-1$ and $i$, for $1\leq i\leq k$, insert the edges between all vertices of each row to form the complete bipartite graph $K_{n,n}$. Define $G^{\circ}_{n,k}$ by identifying the corresponding vertices in rows $0$ and $k$ of $G^{-}_{n,k}$. See Figure~\ref{fig:Gnk}.

Let $G_{n,k}$ denote either $G^{-}_{n,k}$ or $G^{\circ}_{n,k}$, depending on context. Note we consider $G_{n,k}$ to be a graph with labelled vertices.
\end{definition}

\begin{figure}[htbp]
\includegraphics[scale=.85]{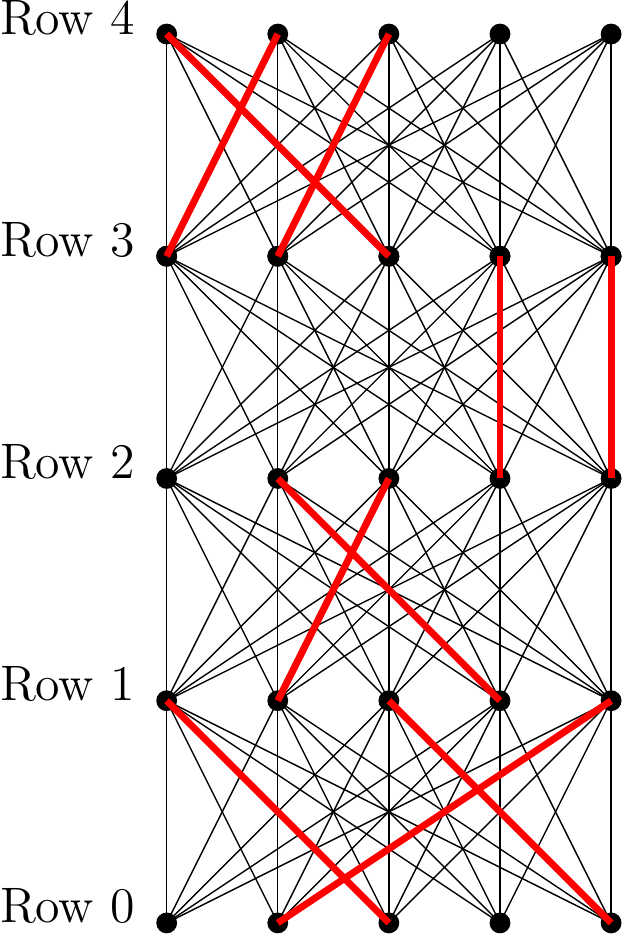}
\hspace{.25in}
\includegraphics[scale=.85]{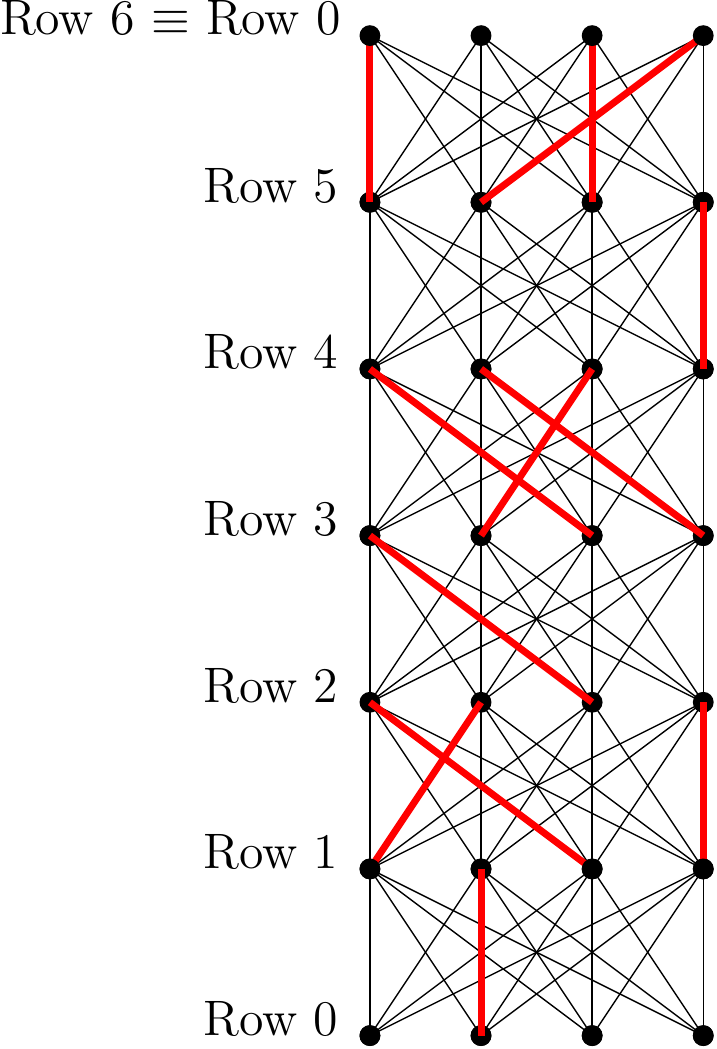}
\caption{Left: A matching on the graph $G_{5,4}^{-}$ corresponding to the maximum rook placement of Figure~\ref{fig:k4n5linrooks}; Right: A perfect matching on the graph $G_{4,6}^{\circ}$ corresponding to the chained permutation of Figure~\ref{fig:k6n4circperm}.}
\label{fig:Gnk}
\end{figure}

\begin{definition}
A \emph{matching} in a graph is a set of edges for which no two share a common vertex. A \emph{perfect matching} is a matching for which each vertex in the graph is incident to exactly one edge in the matching. Note that a necessary condition for a graph to have a perfect matching is that it has an even number of vertices. A \emph{near-perfect matching} of a graph with an odd number of vertices is a matching such that every vertex of the graph except one is incident to an edge in the matching.
\end{definition}

\begin{theorem}
\label{prop:matching}
The set of chained permutations $P^{-}_{n,k}$ is in bijection with perfect matchings on $G^{-}_{n,k}$ if $k$ is odd, and matchings on $G^{-}_{n,k}$ that leave $n$ vertices unmatched if $k$ is even.

The set of chained permutations $P^{\circ}_{n,k}$ is in bijection with perfect matchings on $G^{\circ}_{n,k}$ if at least one of $n$ or $k$ are even, and near-perfect matchings on $G^{\circ}_{n,k}$ if $n$ and $k$ are both odd. 
\end{theorem}

\begin{proof}
The bijection is as follows. Given a chained permutation in $P_{n,k}$, let row $\ell$ of $G_{n,k}$ represent the $\ell$th matrix $X^{(\ell)}$. Construct a matching $M$ of $G_{n,k}$ as follows. If $X_{i,j}^{(\ell)}=1$, include the edge from the $i$th vertex of row $\ell$ to the $j$th vertex of row $\ell-1$ in $M$. Note this is  a matching because if $X_{i,j}^{(\ell)}=1$ then $X_{j,i'}^{(\ell + 1)}\neq 1$ for any $i'$ and $X_{j',i}^{(\ell - 1)}\neq 1$ for any $j'$, so no vertex  of $G_{n,k}$ is incident to more than one edge in $M$. This map is clearly invertible. See Figure~\ref{fig:Gnk} for an example.

By maximality of the sum of the matrix entries, as many vertices as possible are matched. Each $1$ in the chained permutation corresponds to an edge in the matching. So by Lemmas~\ref{prop:maxodd} and \ref{prop:maxeven}, there are $\frac{n(k+1)}{2}$ edges in $M$ for $k$ odd linear and $\lfloor\frac{nk}{2}\rfloor$ in all other cases. But recall, there are $n(k+1)$ vertices in $G^{-}_{n,k}$ and $kn$ vertices in $G^{\circ}_{n,k}$. So $M$ is a perfect matching for $k$ odd linear, since there are $n(k+1)$ vertices in  of $G^{-}_{n,k}$ and $\frac{n(k+1)}{2}$ edges in $M$. For $k$ even linear, there are $n(k+1)$ vertices in $G^{-}_{n,k}$ and $\frac{nk}{2}$ edges in $M$, so $nk$ vertices are incident to an edge in $M$, leaving $n$ vertices unmatched. $M$ is a perfect matching in the circular case for $n$ or $k$ even, since then the number of edges in $M$ is $\lfloor\frac{nk}{2}\rfloor=\frac{nk}{2}$ while the number of vertices in  of $G^{\circ}_{n,k}$ is $nk$. For circular $n$ and $k$ odd, the number of edges in $M$ is $\lfloor\frac{nk}{2}\rfloor=\frac{nk-1}{2}$ while the number of vertices in  of $G^{\circ}_{n,k}$ is $nk$, so the matching is a near-perfect matching since $nk-1$ vertices are incident to an edge in $M$.
\end{proof}

\section{Chained alternating sign matrices}
\emph{Alternating sign matrices} are square matrices with entries in $\{0,1,-1\}$ such that the rows and columns each sum to $1$ and the nonzero entries alternate in sign across each row or column~\cite{MRRASM}; this is a natural superset containing permutations. The enumeration of alternating sign matrices~\cite{zeilberger,kuperbergASMpf}, a major accomplishment in enumerative combinatorics in the 1990's, ignited a flurry of research on the border of algebraic combinatorics and statistical physics, including the proof of the Razumov-Stroganov conjecture~\cite{razstrog,razstrogpf,razstrogpf2}, in addition to much further investigation of  combinatorial properties and connections. 

In Subection~\ref{sec:ASMdef}, we define an alternating sign matrix analogue of chained linear and circular permutations. In Subsection~\ref{sec:enumASM}, we enumerate chained alternating sign matrices for special values of $n$ and $k$. In Subsection~\ref{sec:bijASM}, we draw connections between chained alternating sign matrices and analogues of monotone triangles, square ice, and fully-packed loops on generalized domains. 
\subsection{Definition}
\label{sec:ASMdef}

\begin{definition}
\label{def:chainedASM}
Define \emph{chained (linear or circular) alternating sign matrices} as  $k$-tuples of $n\times n$ $\{-1,0,1\}$--matrices $\left(A^{(1)}, A^{(2)},\ldots,A^{(k)}\right)$ satisfying: 
\begin{enumerate}
\item $\displaystyle\sum_{j=1}^m A^{(\ell)}_{i,j}\in\{0,1\}$ for each $1 \leq i \leq n, 1\leq m \leq n, 1\leq\ell\leq k$,
\item $\displaystyle\sum_{j=1}^n A^{(\ell-1)}_{i,j} + \displaystyle\sum_{j=1}^m A^{(\ell)}_{n+1-j,i}\in\{0,1\}$ 
for each $1 \leq i \leq n, 1 \le m \le n, 1\leq \ell\leq k$, where for linear we consider $A^{(0)}$ to be the zero matrix and for circular we consider $A^{(0)}\equiv A^{(k)}$, and
\item the sum of all entries $\displaystyle\sum_{i,j,\ell}A^{(\ell)}_{i,j}$ is maximized.
\end{enumerate}

Let $ASM_{n,k}^{-}$ denote the set of $k$-chained linear $n\times n$ alternating sign matrices, $ASM_{n,k}^{\circ}$ the set of $k$-chained circular $n\times n$ alternating sign matrices, and $ASM_{n,k}$ either $ASM_{n,k}^{-}$ or $ASM_{n,k}^{\circ}$, depending on context.
\end{definition}

See Figures~\ref{fig:k2n3ASMlin}, \ref{fig:k4n3ASMcirc}, and \ref{fig:boardASM} for examples.

\begin{figure}[htbp]
\begin{center}
\includegraphics[scale=.9]{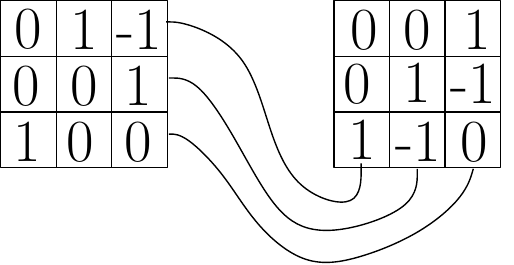}
\end{center}
\caption{A chained alternating sign matrix in $ASM^{-}_{3,2}$.}
\label{fig:k2n3ASMlin}
\end{figure}

\begin{figure}[htbp]
\begin{center}
\includegraphics[scale=1]{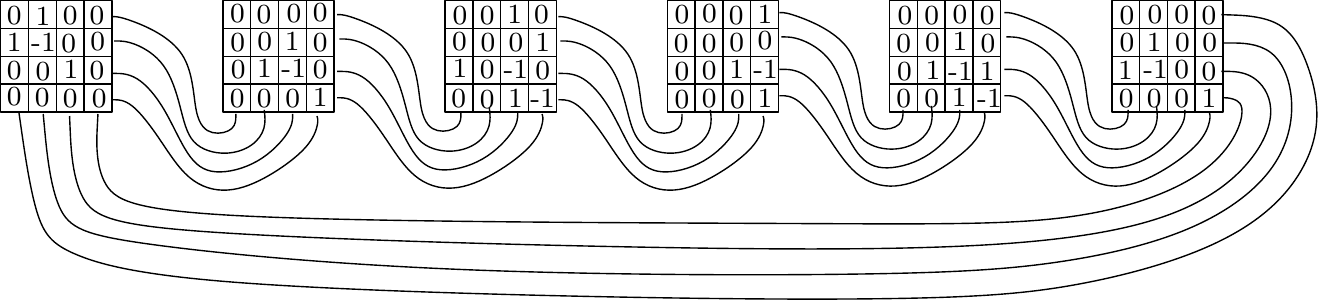}
\end{center}
\caption{A chained alternating sign matrix in $ASM^{\circ}_{4,6}$.}
\label{fig:k4n3ASMcirc}
\end{figure}

\begin{figure}[htbp]
\begin{center}
\includegraphics[scale=.6]{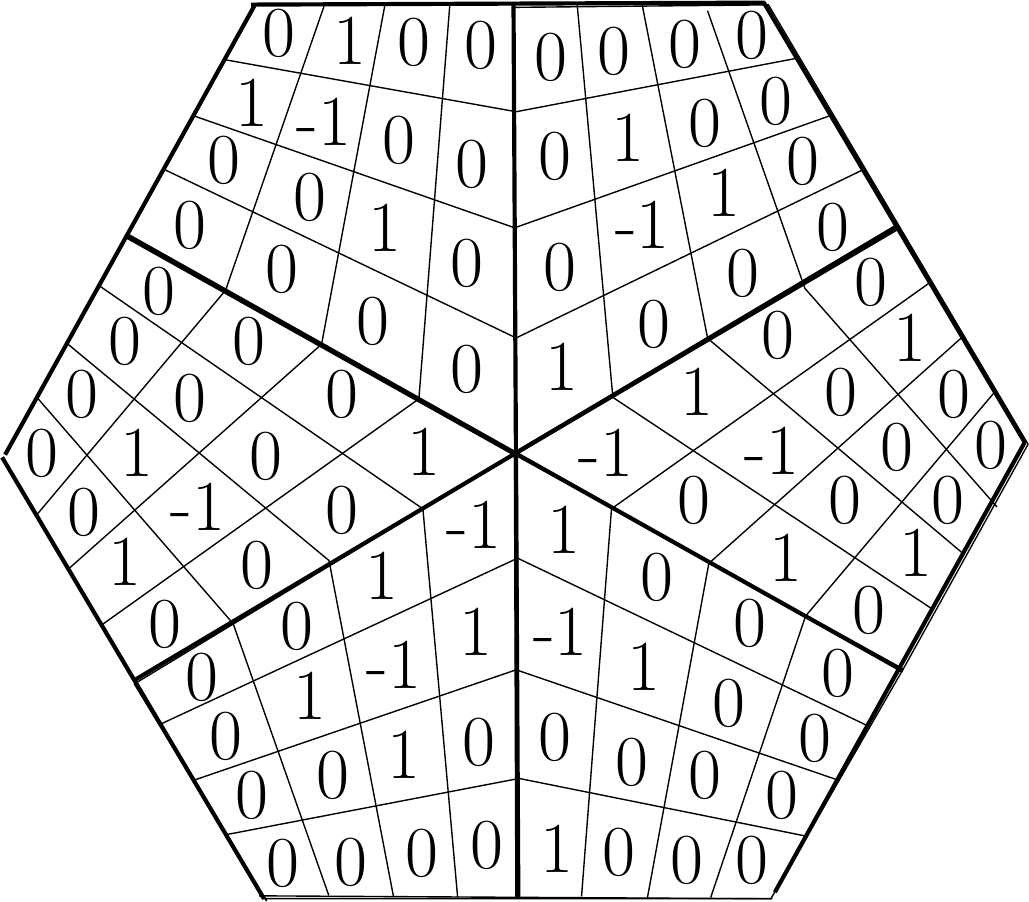}
\end{center}
\caption{The chained alternating sign matrix of Figure~\ref{fig:k4n3ASMcirc}, drawn on the three-person chess board of Figure~\ref{fig:3personchessboard}.}
\label{fig:boardASM}
\end{figure}

\begin{remark}
\label{rmk:asmadj}
It follows from part $(2)$ of Definition~\ref{def:chainedASM} that the total sum of the entries in adjacent matrices in a chained alternating sign matrix is less than or equal to $n$.
\end{remark}

\begin{lemma}
\label{prop:asmsum}
The  sum of entries in a chained alternating sign matrix is the same as the  number of ones in a chained permutation of the same shape. Moreover, a composition $a:=(a_1,a_2,\ldots,a_k)$ 
equals \[\left(\displaystyle\sum_{i,j}A^{(1)},\displaystyle\sum_{i,j}A^{(2)},\ldots,\displaystyle\sum_{i,j}A^{(k)}\right)\]
for some $\left(A^{(1)}, A^{(2)},\ldots,A^{(k)}\right)\in ASM_{n,k}$ if and only if it equals \[\left(\displaystyle\sum_{i,j}X^{(1)},\displaystyle\sum_{i,j}X^{(2)},\ldots,\displaystyle\sum_{i,j}X^{(k)}\right)\] for some $\left(X^{(1)}, X^{(2)},\ldots,X^{(k)}\right)\in P_{n,k}$.
\end{lemma}
\begin{proof} 
The maximum sum is at least the same as in the permutation case, since a chained permutation satisfies Conditions $(1)$ and $(2)$ of Definition~\ref{def:chainedASM}. 

The rest of the claim follows by Remark \ref{rmk:asmadj} and the proof technique of Lemmas~\ref{prop:maxodd} (linear case) and  \ref{prop:maxeven} (circular case) with the following change: instead of placing $\alpha$ rooks on a given board, say board $i$, we have a total sum of $\alpha$ on matrix $A^{(i)}$. 
\end{proof}

\begin{corollary}
\label{cor:nonegative}
The chained alternating sign matrices with no $-1$ entries are exactly the chained permutations.
\end{corollary}

\subsection{Enumeration of special families of chained alternating sign matrices}
\label{sec:enumASM}

In this subsection, we enumerate chained alternating sign matrices for special families of $n$ and $k$. We also present in Table~\ref{table:linearasm} 
some enumeration data for the remaining cases. 

\begin{remark}
\label{rmk:tbls}
In Table~\ref{table:linearasm}, 
we have computed data on the enumeration of chained alternating sign matrices. Note  there may not be a nice product formula for the enumeration in all cases, since, for example, $|A^{-}_{2,6}|=1129$ 
is prime and $|A^{\circ}_{2,8}|=1186=2 \times 593$ 
has a large prime factor.
But in special cases, namely, $k=1$ linear and circular, $k=4$ circular, and $k$ odd linear, we can enumerate $A_{n,k}$ using bijections to objects whose enumerations are known.
\end{remark}

\begin{table}[htbp]
\begin{tabular}{| c | c | c | c | c | c | c | }
  \hline                       
   \diaghead(1,-1)%
   {\theadfont nnn}%
   {k}{n}
   & 1 & 2 & 3 & 4 & 5 & 6 \\
  \hline
  1 & 1 & 2 & 7 & 42 & 429 & 7436 \\
  \hline  
  2 & 2 & 17 & 504 & 53932 & & \\
  \hline  
  3 & 1 & 4 & 49 & & & \\
  \hline  
  4 & 3 & 159 & 98028 & & & \\
  \hline  
  5 & 1 & 8 & & & & \\
  \hline  
  6 & 4 & 1129 & & & & \\
  \hline  
  7 & 1 & 16 & & & &\\
  \hline  
  8 & 5 & 7151 & & & &\\
  \hline
\end{tabular}
\hspace{.1in}
\begin{tabular}{| c | c | c | c | c | c | c | c |}
  \hline                       
   \diaghead(1,-1)%
   {\theadfont nnn}%
   {k}{n}
   & 1 & 2 & 3 & 4 & 5 & 6 \\
  \hline
  1 & 1  & 2 & 20 & 40 & 3430 & 6860 \\
  \hline  
  2 & 2 & 10 & 140 & 5544 & & \\
  \hline  
  3 & 3 & 14 & 3861 &  &  &  \\
  \hline  
  4 & 2 & 42 & 7436 & & &  \\
  \hline  
  5 & 5 & 82 & &  &  &  \\
  \hline  
  6 & 2 & 214 &  &  &  &  \\  
  \hline  
  7 & 7 & 478 &  &  &  &  \\  
  \hline  
  8 & 2 & 1186 &  &  &  &  \\
  \hline  
  9 & 9 & 2786 &  &  &  &  \\
  \hline
\end{tabular}
\caption{Left: Enumeration of chained linear alternating sign matrices for small values of $n$ and $k$. Right: Enumeration of chained circular alternating sign matrices for small values of $n$ and $k$.}
\label{table:linearasm}
\end{table}

\begin{proposition}
\label{prop:asm_n_k1}
$ASM_{n,1}^{-}$
is the set of $n\times n$ alternating sign matrices.
\end{proposition}
\begin{proof}
This follows directly from the definition, since the maximality of Condition (3) of Definition~\ref{def:chainedASM} along with Lemma~\ref{prop:asmsum} implies that the rows and columns  each sum to one.
\end{proof}

We have the following corollary on the cardinality of $ASM_{n,1}^{-}$, which follows from the enumeration of alternating sign matrices~\cite{zeilberger,kuperbergASMpf}.
\begin{corollary}
$|ASM_{n,1}^{-}|=\displaystyle\prod_{k=0}^{n-1} \frac{\left(3k+1\right)!}{\left(n+k\right)!}$.
\end{corollary}

For $n=1$, we reduce to the permutation case.
\begin{proposition}
$ASM_{1,k} = P_{1,k}$
\end{proposition}
\begin{proof}
By Property $(1)$ of Definition~\ref{def:chainedASM}, no $-1$ is allowed to be in the leftmost column of any of the matrices in a chained alternating sign matrix. In $ASM_{1,k}$, the matrices each have only one column, so none may include a $-1$. Thus, the claim follows from Corollary~\ref{cor:nonegative}.
\end{proof}

For odd $k$ values, we have the following theorem.
\begin{theorem}
\label{thm:ASM_koddlin}
For $k$ odd, $ASM_{n,k}^{-}$ 
is in bijection with the set of $\frac{k+1}{2}$-tuples of $n\times n$ alternating sign matrices.
\end{theorem}
\begin{proof}
Let $k$ be odd and $A=\left(A^{(1)}, A^{(2)},\ldots,A^{(k)}\right)\in ASM_{n,k}^{-}$. By Lemmas \ref{prop:asmsum} and \ref{prop:maxodd}, $$\left(\displaystyle\sum_{i,j}A^{(1)},\displaystyle\sum_{i,j}A^{(2)},\ldots,\displaystyle\sum_{i,j}A^{(k)}\right)=(n,0,n,\dots,0,n).$$ We wish to show that each matrix of even index is an ASM and each matrix of odd index is the all zeros matrix.
  Suppose to the contrary there is an even numbered matrix that is not all zeros, say $A^{(\ell)}$. Then $A^{(\ell)}$ must have a 1. On $A^{(\ell)}$, find the leftmost column with a 1 in it and pick the bottommost 1 in this column. Say this 1 is in column $i$, row $j$. Picking such a 1 guarantees that there is no $-1$ below it in column $i$.  Therefore, row $i$ on $A^{(\ell-1)}$ must sum to 0. As a result, the sum of entries on $A^{(\ell-1)}$ is less than or equal to $n-1$, a contradiction. Therefore, all even indexed matrices must contain all 0s and all odd indexed matrices must be alternating sign matrices.
\end{proof}

\begin{corollary}
For $k$ odd, $|ASM_{n,k}^{-}|=\displaystyle\left(\prod_{k=0}^{n-1} \frac{\left(3k+1\right)!}{\left(n+k\right)!}\right)^{\frac{k+1}{2}}$.
\end{corollary}

We make two more observations about chained circular alternating sign matrices for particular values of $n$ or $k$.

\begin{proposition}
\label{prop:ASMk4}
$ASM_{n,4}^{\circ}$
is in bijection with the set of $2n \times 2n$ alternating sign matrices.
\end{proposition}
\begin{proof}
Let $A=\left(A^{(1)}, A^{(2)},A^{(3)},A^{(4)}\right)\in ASM_{n,4}^{\circ}$.
By Lemmas \ref{prop:asmsum} and \ref{prop:maxeven}, $$\left(\displaystyle\sum_{i,j}A^{(1)},\displaystyle\sum_{i,j}A^{(2)},\displaystyle\sum_{i,j}A^{(3)},\displaystyle\sum_{i,j}A^{(4)}\right)=(n-\alpha,\alpha,n-\alpha,\alpha)$$ for some $0\leq \alpha\leq n$.
In particular, $\displaystyle\sum_{i,j}A^{(\ell-1)}+\displaystyle\sum_{i,j}A^{(\ell)}=n$ for all $1\leq \ell\leq 4$ implies that $\displaystyle\sum_{j=1}^n A^{(\ell-1)}_{i,j} + \displaystyle\sum_{j=1}^n A^{(\ell)}_{n+1-j,i}=1$ for all $1\leq \ell\leq 4$. We form a $2n\times 2n$ matrix $M$ by concatenating $A^{(1)},\ldots,A^{(4)}$ as follows: let $M_{ij}=A^{(1)}_{ij}$ for $1\leq i,j\leq n$. Let $M_{ij}$ for $1\leq i\leq n$ and $n+1\leq j\leq 2n$ be the entries of $A^{(2)}$ rotated a quarter turn clockwise. Let $M_{ij}$ for $n+1\leq i,j\leq 2n$ be the entries of $A^{(3)}$ rotated a half turn. Finally,  let $M_{ij}$ for $n+1\leq i\leq 2n$ and $1\leq j\leq n$ be the entries of $A^{(4)}$ rotated a quarter turn counterclockwise. $M$ is an alternating sign matrix, since, by construction, the rows and columns each sum to 1, and by part $(2)$ of Definition~\ref{def:chainedASM}, the nonzero entries alternate in sign across each row or column.
This construction is clearly invertible and is thus a bijection.
\end{proof}

\begin{corollary}
$|ASM_{n,4}^{\circ}|=\displaystyle\prod_{k=0}^{2n-1} \frac{\left(3k+1\right)!}{\left(2n+k\right)!}$.
\end{corollary}

Chained circular alternating sign matrices with $k=1$ are related to a \emph{symmetry class} of alternating sign matrices, namely, quarter-turn symmetric alternating sign matrices. These are alternating sign matrices that are invariant under $90\degree$ rotation. Many symmetry classes of alternating sign matrices, including quarter-turn symmetric, are enumerated by nice product formulas; see~\cite{Kuperberg}. See Figures~\ref{fig:QTASM} and \ref{fig:12x12matrix} for an example related to the following proposition.

\begin{proposition}
\label{prop:ASMn1circ}
$ASM_{n,1}^{\circ}$ with $n$ even
is in bijection with the set of quarter-turn symmetric alternating sign matrices of size $2n \times 2n$.
\end{proposition}

\begin{proof}
Let $A=\left(A^{(1)}\right)\in ASM_{n,1}^{\circ}$.
By Lemmas \ref{prop:asmsum} and \ref{prop:maxeven}, $\displaystyle\sum_{i,j}A^{(1)}=\frac{n}{2}$.
In particular, $\displaystyle\sum_{j=1}^n A^{(1)}_{i,j} + \displaystyle\sum_{j=1}^n A^{(1)}_{n+1-j,i}=1$ for all $1\leq i\leq n$. Form a $2n\times 2n$ matrix $M$ by the same construction as in Proposition~\ref{prop:ASMk4}, using four copies of $A^{(1)}$. $M$ is quarter-turn symmetric  by construction and is an alternating sign matrix since the rows and columns each sum to 1, and by part $(2)$ of Definition~\ref{def:chainedASM}, the nonzero entries alternate in sign across each row or column.
This construction is clearly invertible and is thus a bijection.
\end{proof}

\begin{corollary}
\label{cor:qtasm}
\scalebox{.96}{$|ASM_{2m,1}^{\circ}|=\displaystyle\left(\prod_{k=0}^{m-1} \frac{\left(3k+1\right)!}{\left(m+k\right)!}\right)^3 \  \prod_{i=1}^{m}\left(\frac{3i-1}{3i-2}\prod_{j=i}^{m}\frac{m+i+j-1}{2i+j-1}\right)$.}
\end{corollary}

\begin{proof}
This follows from Proposition~\ref{prop:ASMn1circ} and the enumeration of quarter-turn symmetric alternating sign matrices due to Kuperberg~\cite{Kuperberg}.
\end{proof}

\begin{figure}[htbp]
\centering
\includegraphics[scale=.85]{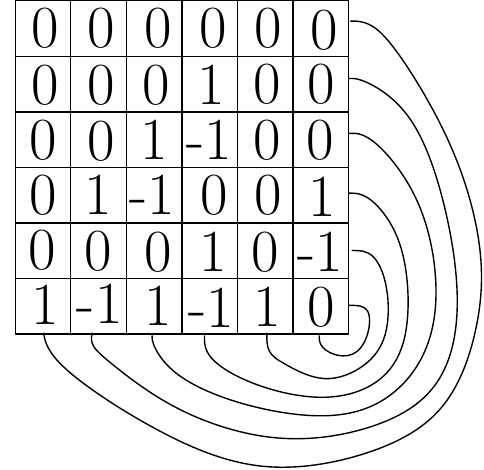}
\caption{A chained alternating sign matrix in $ASM^{\circ}_{6,1}$.}
\label{fig:QTASM}
\end{figure}

\begin{figure}[htbp]
\centering
$\left[\begin{array}{rrrrrrrrrrrr}
0 & 0 & 0 & 0 & 0 & 0 & 1 & 0 & 0 & 0 & 0 & 0 \\
0 & 0 & 0 & 1 & 0 & 0 & -1 & 0 & 1 & 0 & 0 & 0\\
0 & 0 & 1 & -1 & 0 & 0 & 1 & 0 & -1 &1 & 0 & 0\\
0 & 1 & -1 & 0 & 0 & 1 & -1 & 1 & 0 & -1 & 1 & 0\\
0 & 0 & 0 & 1 & 0 & -1 & 1 & 0 & 0 & 0 & 0 & 0\\
1 & -1 & 1 & -1 & 1 & 0 & 0 & -1 & 1 & 0 & 0 &  0\\
0 & 0 & 0 & 1 & -1 & 0 & 0 & 1 & -1 & 1 & -1 & 1\\
0 & 0 & 0 & 0 & 0 & 1 & -1 & 0 & 1 & 0 & 0 & 0\\
0 & 1 & -1 & 0 & 1 & -1 & 1 & 0 & 0 & -1 & 1 & 0\\
0 & 0 & 1 & -1 & 0 & 1 & 0 & 0 & -1 & 1 & 0 & 0\\
0 & 0 & 0 & 1 & 0 & -1 & 0 & 0 & 1 & 0 & 0 & 0\\
0 & 0 & 0 & 0 & 0 & 1 & 0 & 0 & 0 & 0 & 0 & \enskip 0
\end{array}\right]$
\caption{The $12 \times 12$ quarter-turn symmetric alternating sign matrix corresponding to the chained alternating sign matrix of Figure~\ref{fig:QTASM}.}
\label{fig:12x12matrix}
\end{figure}

\subsection{Chained alternating sign matrix bijections}
\label{sec:bijASM}
In the spirit of~\cite{manyfaces}, we transform chained alternating sign matrices into other forms, namely, the analogues of monotone triangles, square ice configurations, and fully-packed loops. 
We concentrate on the circular case, since in the linear case there may be some negative one entries in the top row (see Figure~\ref{fig:k2n3ASMlin}), which would cause complications in or failures of these constructions. 

\begin{definition}
\label{def:mt} 
Let $k$ be even and $A=\left(A^{(1)},A^{(2)},\ldots,A^{(k)}\right)\in ASM_{n,k}^{\circ}$. 
For each pair of matrices $\left(A^{(2\ell-1)},A^{(2\ell)}\right)$, consider the $n\times 2n$ matrix $B^{(\ell)}$ defined by concatenating $A^{(2\ell-1)}$ with the quarter turn clockwise rotation of $A^{(2\ell)}$. 
We then apply the standard monotone triangle map to each $B^{(\ell)}$ to create an array of numbers 
$M^{(\ell)}$. 
Namely, let the entries in in row $m$ of $M^{(\ell)}$ be all the $j$ such that the column partial sum $\displaystyle\sum_{i=1}^{m} B^{(\ell)}_{i,j}$ is equal to $1$. Order entries in each row of $M^{(\ell)}$ to be increasing. 
We call
$\left(M^{(1)},M^{(2)},\ldots,M^{(\frac{k}{2})}\right)$
the \emph{chained monotone triangle} corresponding to the chained alternating sign matrix $\left(A^{(1)},A^{(2)},\ldots,A^{(k)}\right)$.
\label{nx2nASM}
\end{definition}

\begin{example}
\label{ex:Bell}
The $4 \times 8$ matrices $B^{(1)},B^{(2)},B^{(3)}$ from Definition~\ref{nx2nASM} that correspond to the chained alternating sign matrix of Figure~\ref{fig:k4n3ASMcirc} are:

\vspace{1ex}
\begin{center}
\scalebox{.73}{
$\left[\begin{array}{rrrrrrrr} 0&1&\; 0&\; 0&\; 0&0&\; 0&\; 0\\1&-1&0&0&0&1&0&0\\0&0&1&0&0&-1&1&0\\0&0&0&0&1&0&0&0 \end{array}\right] \; \left[\begin{array}{rrrrrrrr} 0&\;0&1&0&\;0&\;0&\;0&\;0\\0&0&0&1&0&0&0&0\\1&0&-1&0&0&1&0&0\\0&0&1&-1&1&0&0&0 \end{array} \right] \; \left[\begin{array}{rrrrrrrr} 0&\;0&0&0&\;0&1&\;0&\;0\\0&0&1&0&0&-1&1&0\\0&1&-1&1&0&1&0&0\\0&0&1&-1&1&0&0&0 \end{array} \right]$}.
\end{center}

\vspace{1ex}
Then the corresponding chained monotone triangle is:

\vspace{1ex}
\begin{center}
$\left(\begin{array}{lllllll}
&&&2&&& \\
&&1&&6&& \\
&1&&3&&7& \\
1&&3&&5&&7
\end{array} \ , \
\begin{array}{lllllll}
&&&3&&& \\
&&3&&4&& \\
&1&&4&&6& \\
1&&3&&5&&8
\end{array} \ , \
\begin{array}{lllllll}
&&&6&&& \\
&&3&&7&& \\
&2&&4&&7& \\
2&&3&&5&&7
\end{array}\right)$.
\end{center}
\end{example}

\vspace{1ex}
Chained monotone triangles may be described without reference to chained alternating sign matrices as follows. We first need the following definition.

\begin{definition}
\label{def:gt}
A \emph{Gelfand-Tsetlin pattern} of \emph{order} $n$ is a triangular array $t_{ij}$ with $1\leq i\leq n$, $1\leq j\leq i$, such that $t_{i+1,j}\leq t_{ij}\leq t_{i+1,j+1}$. A Gelfand-Tsetlin pattern is \emph{strict} if, in addition, $t_{ij} < t_{i,j+1}$.
\end{definition}

\begin{theorem}
\label{prop:mt} Let $k$ be even.
$\left(M^{(1)},M^{(2)},\ldots,M^{(\frac{k}{2})}\right)$ is a chained monotone triangle corresponding to a chained alternating sign matrix in $ASM_{n,k}^{\circ}$ if and only if:
\begin{enumerate}
\item Each $M^{(\ell)}$ is a strict Gelfand-Tsetlin pattern of order $n$, and
\item For any $1\leq\ell\leq\frac{k}{2}$, there is no number $i\leq n$ such that the following are both true:
\begin{itemize} \item $i$ appears in the largest row of $M^{(\ell)}$, and 
\item $2n-i+1$ appears in the largest row of $M^{(\ell-1)}$, where we consider $M^{(0)}\equiv M^{(\frac{k}{2})}$.
\end{itemize}
\end{enumerate}
\end{theorem}

\begin{proof} Let $k$ be even and  $\left(A^{(1)},A^{(2)},\ldots,A^{(k)}\right)\in ASM^{\circ}_{n,k}$. We show each $M^{(\ell)}$ is a Gelfand-Tsetlin pattern.
From Lemmas~\ref{prop:asmsum} and \ref{prop:maxeven}, we know that $\displaystyle\sum_{i,j}A^{(2\ell-1)}+\displaystyle\sum_{i,j}A^{(2\ell)}=n$.  Also, by Conditions $(1)$ and $(2)$ of Definition~\ref{def:chainedASM}, the column partial sums of $B^{(\ell)}$ are zero or one and there are $i$ columns in row $i$ of $B^{(\ell)}$ which have a partial sum from the top of one. $M^{(\ell)}$ is \emph{strict} by construction, since its rows are strictly increasing. So Condition $(1)$ is satisfied. 

Condition $(2)$ says that if the sum of column $i$ in $A^{(2\ell-1)}$ is one, then the sum of row $i$ in $A^{(2\ell-2)}$ is zero and if the sum of row $i$ in $A^{(2\ell-2)}$ is one, then the sum of column $i$ in $A^{(2\ell-1)}$ is zero. This is true by part $(2)$ of Definition~\ref{def:chainedASM}.

Given $\left(M^{(1)},M^{(2)},\ldots,M^{(\frac{k}{2})}\right)$ satisfying the above conditions, we may reconstruct the chained alternating sign matrix by inverting the map described in Definition~\ref{def:mt}. Thus, this is a bijection.
\end{proof}

We now define the chained grid graph, which we  use in the definitions of both chained ice configurations and chained fully-packed loops.

\begin{definition}
\label{def:gg}
Let $k$ be even.
Define the chained grid graph $GG_{n,k}$ as follows. Let there be \emph{interior vertices} $v_{i,j}^{(\ell)}$ for $1\leq i,j\leq n$ and $1\leq\ell \leq k$ and \emph{boundary vertices} $v_{0,j}^{(\ell)}$ and $v_{i,0}^{(\ell)}$ for all $1\leq i,j\leq n$. Let there be the following edges:
\begin{enumerate}
\item \emph{interior horizontal edges} between $v_{i,j}^{(\ell)}$ and $v_{i,j+1}^{(\ell)}$ for all $1\leq i\leq n$, $1\leq \ell\leq k$, and $1\leq j<n$, 
\item \emph{interior vertical edges} between $v_{i,j}^{(\ell)}$ and $v_{i+1,j}^{(\ell)}$ for all $1\leq j\leq n$, $1\leq \ell\leq k$, and $1\leq i<n$,
\item \emph{chaining edges} between $v_{i,n}^{(\ell)}$ and $v_{n,i}^{(\ell+1)}$ for all $1\leq i\leq n$, $1\leq \ell\leq k$, where $k+1\equiv 1$, and
\item \emph{boundary edges} between $v_{0,j}^{(\ell)}$ and $v_{1,j}^{(\ell)}$ and between $v_{i,0}^{(\ell)}$ and $v_{i,1}^{(\ell)}$ for all $1\leq i,j\leq n$, $1\leq \ell\leq k$.
\end{enumerate}
\end{definition}

\begin{figure}[htbp]
\begin{center}
\includegraphics[scale=1]{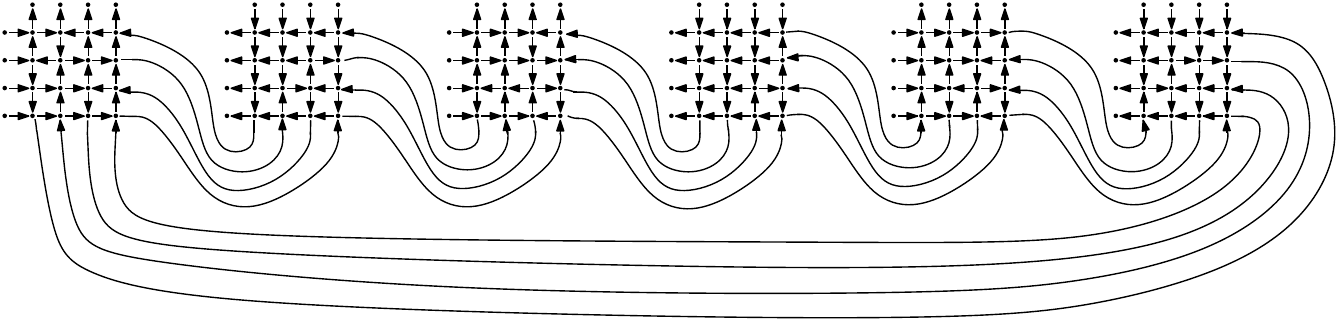}
\end{center}
\caption{A chained ice configuration corresponding to the chained alternating sign matrix of Figure~\ref{fig:k4n3ASMcirc}.}
\label{fig:chained_sq_ice}
\end{figure}

\begin{definition}
\label{def:chained_ice}
Let $k$ be even and $A\in ASM_{n,k}^{\circ}$. Then the \emph{chained  ice configuration} corresponding to $A$ is a directed graph with $GG_{n,k}$ as its underlying undirected graph and the direction of each edge determined by the following conditions.
\begin{enumerate}
\item Interior horizontal edges are directed as follows for all $1\leq i\leq n$, $1\leq j<n$, $1\leq \ell\leq \frac{k}{2}$: 
\begin{align*}
v_{i,j}^{(2\ell-1)} &\leftarrow v_{i,j+1}^{(2\ell-1)} \mbox{ if } \displaystyle\sum_{j_0=1}^j A^{(2\ell-1)}_{i,j_0}=1, \hspace{1cm} v_{i,j}^{(2\ell-1)} \rightarrow v_{i,j+1}^{(2\ell-1)} \mbox{ if } \displaystyle\sum_{j_0=1}^j A^{(2\ell-1)}_{i,j_0}=0, \\
v_{i,j}^{(2\ell)} &\rightarrow v_{i,j+1}^{(2\ell)} \mbox{ if } \displaystyle\sum_{j_0=1}^j A^{(2\ell)}_{i,j_0}=1, \hspace{1.9cm}
 v_{i,j}^{(2\ell)} \leftarrow v_{i,j+1}^{(2\ell)} \mbox{ if } \displaystyle\sum_{j_0=1}^j A^{(2\ell)}_{i,j_0}=0.
\end{align*}
\item Interior vertical edges are directed as follows for all $1\leq i< n$, $1\leq j\leq n$, $1\leq \ell\leq \frac{k}{2}$: 
\begin{align*}
v_{i,j}^{(2\ell-1)} &\leftarrow v_{i+1,j}^{(2\ell-1)} \mbox{ if }  \displaystyle\sum_{j_0=1}^n A^{(2\ell-2)}_{j,j_0} + \displaystyle\sum_{j_0=1}^{n-i} A^{(2\ell-1)}_{n+1-j_0,j}=1,\\
v_{i,j}^{(2\ell-1)} &\rightarrow v_{i+1,j}^{(2\ell-1)} \mbox{ if }  \displaystyle\sum_{j_0=1}^n A^{(2\ell-2)}_{j,j_0} + \displaystyle\sum_{j_0=1}^{n-i} A^{(2\ell-1)}_{n+1-j_0,j}=0, 
\\
v_{i,j}^{(2\ell)} &\rightarrow v_{i+1,j}^{(2\ell)} \mbox{ if }  \displaystyle\sum_{j_0=1}^n A^{(2\ell-1)}_{j,j_0} + \displaystyle\sum_{j_0=1}^{n-i} A^{(2\ell)}_{n+1-j_0,j}=1, \\ 
v_{i,j}^{(2\ell)} &\leftarrow v_{i+1,j}^{(2\ell)} \mbox{ if }  \displaystyle\sum_{j_0=1}^n A^{(2\ell-1)}_{j,j_0} + \displaystyle\sum_{j_0=1}^{n-i} A^{(2\ell)}_{n+1-j_0,j}=0.
\end{align*}
\item Chaining edges are directed as follows for all $1\leq i\leq n$, $1\leq \ell\leq \frac{k}{2}$: 
\begin{align*} v_{i,n}^{(2\ell-1)} &\leftarrow v_{n,i}^{(2\ell)} \mbox{ if } \displaystyle\sum_{j_0=1}^n A^{(2\ell-1)}_{i,j_0}=1, \hspace{1cm}
v_{i,n}^{(2\ell-1)} \rightarrow v_{n,i}^{(2\ell)} \mbox{ if } \displaystyle\sum_{j_0=1}^n A^{(2\ell-1)}_{i,j_0}=0, \\
v_{i,n}^{(2\ell)} &\rightarrow v_{n,i}^{(2\ell+1)} \mbox{ if } \displaystyle\sum_{j_0=1}^n A^{(2\ell)}_{i,j_0}=1, \hspace{1cm}
v_{i,n}^{(2\ell)} \leftarrow v_{n,i}^{(2\ell+1)} \mbox{ if } \displaystyle\sum_{j_0=1}^n A^{(2\ell)}_{i,j_0}=0.
\end{align*}
\item Boundary edges are directed as follows for all $1\leq i,j\leq n$, $1\leq \ell\leq \frac{k}{2}$: 

\begin{center}
$v_{i,0}^{(2\ell-1)}\rightarrow v_{i,1}^{(2\ell-1)}$, \   $v_{0,j}^{(2\ell-1)} \leftarrow v_{1,j}^{(2\ell-1)}$, \   $v_{i,0}^{(2\ell)}\leftarrow v_{i,1}^{(2\ell)}$, \ and  $v_{0,j}^{(2\ell)} \rightarrow v_{1,j}^{(2\ell)}$. 
\end{center}

\noindent
Call these \emph{chained domain wall boundary conditions}.
\end{enumerate}
\end{definition}

See Figure~\ref{fig:chained_sq_ice} for an example.

\vspace{1ex}
Chained ice configurations may be described without reference to chained alternating sign matrices as in the following theorem. The proof is rather technical, so we postpone it to the appendix.

\begin{theorem}
\label{prop:sqice}
A directed graph with underlying graph $GG_{n,k}$, for some $n$ and even $k$, is a chained ice configuration corresponding to a chained alternating sign matrix in $ASM_{n,k}^{\circ}$ if and only if it has chained domain wall boundary conditions ((4) in Definition~\ref{def:chained_ice}) and each interior vertex has two edges entering and two edges leaving. That is, each interior vertex is in one of the six configurations in Figure~\ref{fig:6V}.
\end{theorem}

\begin{figure}[htbp]
\begin{center}
\includegraphics[scale=1]{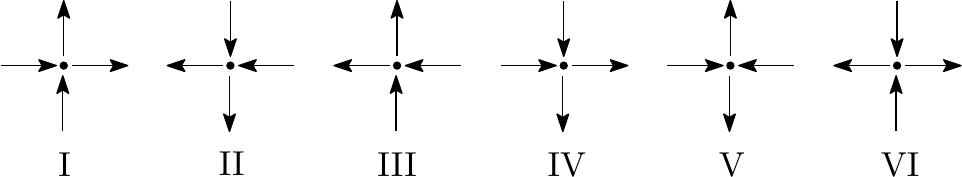}
\end{center}
\caption{The six vertex configurations of Theorem~\ref{prop:sqice}.}
\label{fig:6V}
\end{figure}

We now define the chained analogue of fully-packed loop configurations.
\begin{definition}
\label{def:fpl}
Let $k$ be even and $A\in ASM_{n,k}^{\circ}$.
Consider the corresponding chained ice configuration. Say the vertex $v_{i,j}^{(\ell)}$ has \emph{parity} equal to the parity of $i+j+\ell$. Pick the directed edges that point from an even vertex to an odd vertex; make these undirected edges in a new graph with the same vertices. We call this the \emph{chained fully-packed loop configuration} corresponding to the chained alternating sign matrix. 
\end{definition}

See Figures~\ref{fig:fpl_new} and~\ref{fig:FPL} for an example.

\begin{figure}[htbp]
\includegraphics[scale=1]{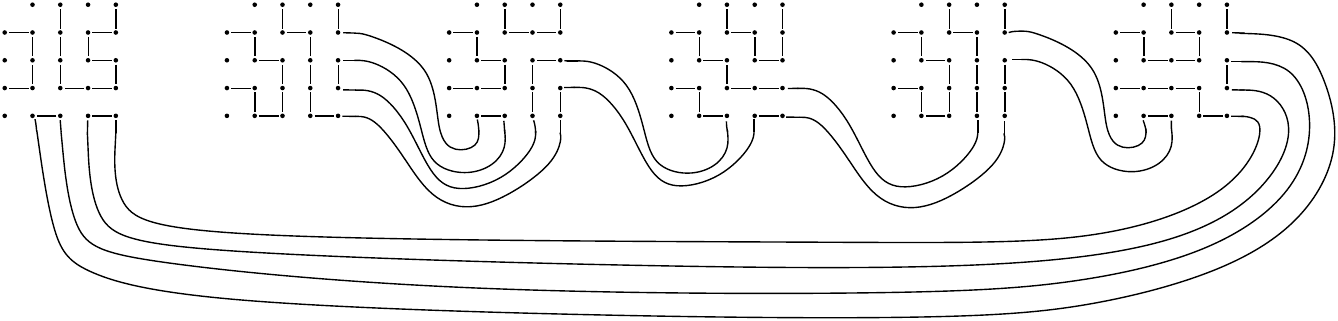}
\caption{The fully-packed loop configuration corresponding to the chained alternating sign matrix of Figure~\ref{fig:boardASM} and the chained ice configuration of Figure~\ref{fig:chained_sq_ice}.}
\label{fig:fpl_new}
\end{figure}

\begin{figure}[htbp]
\begin{center}
\includegraphics[scale=.6]{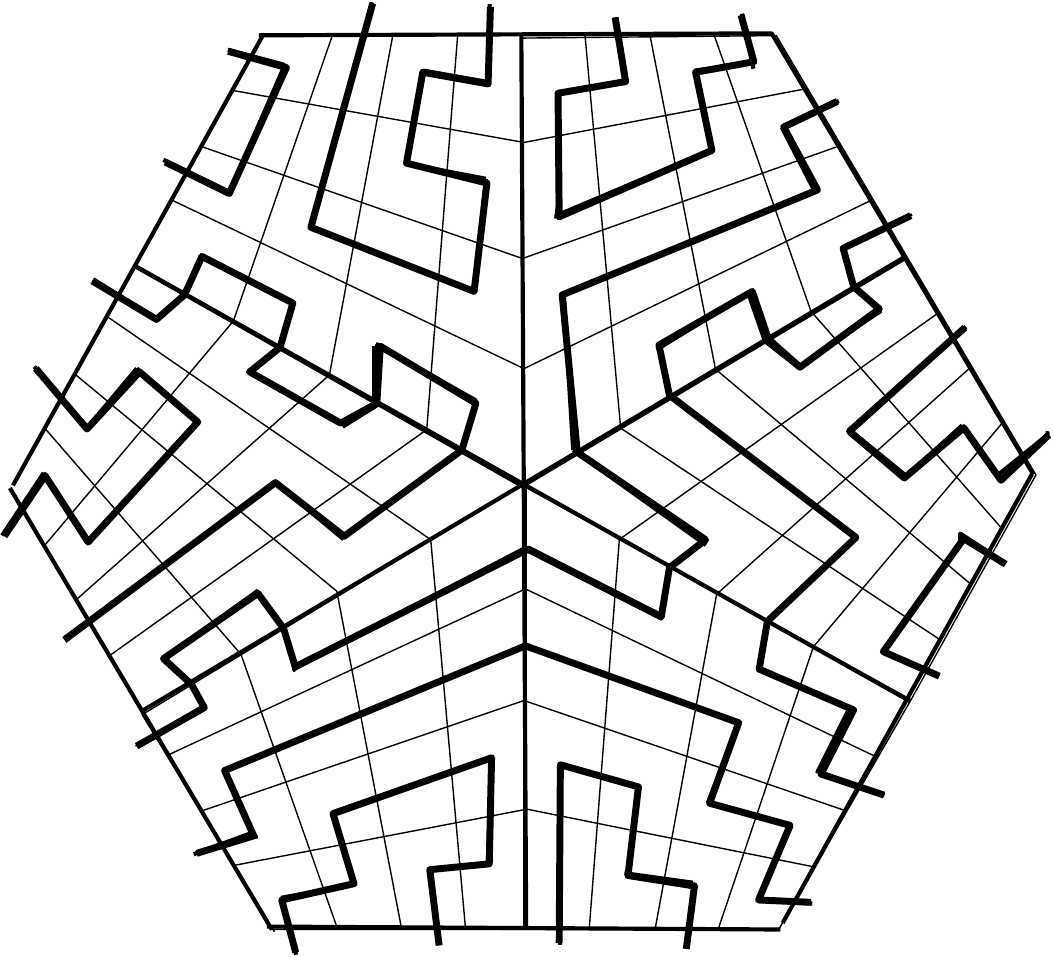}
\caption{The fully-packed loop configuration of Figure~\ref{fig:fpl_new}, drawn on the three-person chessboard of Figure~\ref{fig:3personchessboard}.}
\label{fig:FPL}
\end{center}
\end{figure}

\vspace{1ex}
Chained fully-packed loop configurations may be described without reference to chained ice configurations as follows.

\begin{theorem}
\label{prop:fpl}
A subgraph of $GG_{n,k}$ is a chained fully-packed loop configuration if and only if it contains 
the boundary edges between $v_{i,0}^{(\ell)}$ and $v_{i,1}^{(\ell)}$ whenever $i$ is odd and between $v_{0,j}^{(\ell)}$ and $v_{1,j}^{(\ell)}$ whenever $j$ is even
and its interior vertices are each adjacent to exactly two edges.
\end{theorem}

\begin{proof}
Let $A=\left(A^{(1)},A^{(2)},\ldots,A^{(k)}\right)\in ASM_{n,k}^{\circ}$ be a chained alternating sign matrix. Let $G$ be its corresponding chained ice configuration, constructed as in Definition~\ref{def:chained_ice}, and $F$ its corresponding chained fully-packed loop, constructed as in Definition~\ref{def:fpl}. By the chained domain wall boundary conditions of $G$,  there is a directed edge from $v_{i,0}^{(\ell)}$ to $v_{i,1}^{(\ell)}$ if $\ell$ is odd (so in this case $\ell$+$i$ is even) and there is a directed edge from $v_{i,1}^{(\ell)}$ to $v_{i,0}^{(\ell)}$ if $\ell$ is even (so in this case $\ell+i+1$ is even). Thus, $F$ contains edges between $v_{i,0}^{(\ell)}$ and $v_{i,1}^{(\ell)}$ whenever $i$ is odd.
Also by the boundary conditions, there is a directed edge from $v_{1,j}^{(\ell)}$ to $v_{0,j}^{(\ell)}$ if $\ell$ is odd (so in this case $\ell+j+1$ is even) and there is a directed edge from $v_{0,j}^{(\ell)}$ to $v_{1,j}^{(\ell)}$ if $\ell$ is even (so in this case $\ell+j$ is even). Thus, $F$ contains edges between $v_{0,j}^{(\ell)}$ and $v_{1,j}^{(\ell)}$ whenever $j$ is even. By Theorem~\ref{prop:sqice}, each interior vertex of $G$ has two edges directed inward and two edges directed outward. Thus, each interior vertex of $F$ is adjacent to exactly two edges. Therefore, $F$ satisfies the conditions of the theorem.

Conversely, let $F$ be a subgraph of $GG_{n,k}$ that contains boundary edges between $v_{i,0}^{(\ell)}$ and $v_{i,1}^{(\ell)}$ whenever $i$ is odd and between $v_{0,j}^{(\ell)}$ and $v_{1,j}^{(\ell)}$ whenever $j$ is even and whose interior vertices are adjacent to exactly two edges. We may construct the corresponding chained square ice configuration by inverting the map described in Definition~\ref{def:fpl}. Thus, this is a bijection.
\end{proof}

\begin{remark}
We note that these chained fully-packed loop configurations are some of the generalized domains considered by Cantini and Sportiello in their refined proof of the Razumov-Stroganov conjecture~\cite{razstrogpf2}. Key to this proof was the fact that the action of \emph{gyration} is well-defined on these domains and rotates the \emph{link pattern} in the same way as on fully-packed loops on the square grid (proved in~\cite{wieland}). See also~\cite{razstrogrow}.
\end{remark}

\appendix
\section{Proof of Theorem \ref{prop:sqice}}
The proof of Theorem \ref{prop:sqice} is rather lengthy and technical. We opted to provide the entire proof for the sake of clarity, rather than leaving some cases to the reader.

\begin{proof}[Proof of Theorem \ref{prop:sqice}]
Let $A=\left(A^{(1)},A^{(2)},\ldots,A^{(k)}\right)\in ASM_{n,k}^{\circ}$ be a chained alternating sign matrix. Consider its corresponding chained ice configuration, constructed as in Definition~\ref{def:chained_ice}.  

Condition (4) of Definition~\ref{def:chained_ice} is the chained domain wall boundary conditions. It remains to show that 
each interior vertex has two edges entering and two edges leaving, that is, each interior vertex is in one of the configurations of Figure~\ref{fig:6V}. 

Suppose $1\leq i,j\leq n$. 
For ease of notation, since we are in $ASM_{n,k}^{\circ}$, we consider $A^{(m)}\equiv A^{(m+n)}$ for all $m$.
Consider the configuration at vertex $v_{i,j}^{(\ell)}$ for some $1\leq i,j\leq n$ and $1\leq \ell\leq k$. Denote as $N$ the edge between $v_{i,j}^{(\ell)}$ and $v_{i-1,j}^{(\ell)}$, as $S$ the edge between $v_{i,j}^{(\ell)}$ and $v_{i+1,j}^{(\ell)}$, as $E$ the edge between $v_{i,j}^{(\ell)}$ and $v_{i,j+1}^{(\ell)}$, and as $W$ the edge between $v_{i,j}^{(\ell)}$ and $v_{i,j-1}^{(\ell)}$, whenever these edges are defined. If $i=n$, then $S$ is the chaining edge between $v_{n,j}^{(\ell)}$ and $v_{j,n}^{(\ell-1)}$. If $j=n$, then $E$ is the chaining edge between $v_{i,n}^{(\ell)}$ and $v_{n,i}^{(\ell+1)}$. 

\vspace{1ex}
\textbf{Case $A_{i,j}^{(\ell)}=0$:} 

If $1<j<n$, then by Definition~\ref{def:chainedASM} we know
$\displaystyle\sum_{j_0=1}^j A^{(\ell)}_{i,j_0}=\displaystyle\sum_{j_0=1}^{j-1} A^{(\ell)}_{i,j_0}$, so 
by $(1)$ of Definition~\ref{def:chained_ice}, $W$ and $E$ are
either both directed left or both directed right. 

If $j=1$ and $\ell$ is  odd, then $W$ is a boundary edge that is directed right by $(4)$, and by $(1)$, $E$ is directed to the right as well.

If $j=1$ and $\ell$ is even, then $W$ is a boundary edge that is directed left by $(4)$, and by $(1)$, $E$ is directed to the left as well.

If $j=n$ and $\ell$ is odd, then $E$ is a chaining edge that by $(3)$ is directed left if $\displaystyle\sum_{j_0=1}^n A^{(\ell)}_{i,j_0}=1$ and right if 
the sum is $0$. By $(1)$, $W$ is directed left if $\displaystyle\sum_{j_0=1}^{n-1} A^{(\ell)}_{i,j_0}=1$ and right if 
the sum if $0$. We know  $\displaystyle\sum_{j_0=1}^{n-1} A^{(\ell)}_{i,j_0}= \displaystyle\sum_{j_0=1}^{n} A^{(\ell)}_{i,j_0}$, so $W$ and $E$ are either both directed left or both directed right.

If $j=n$ and $\ell$ is  even, then $E$ is a chaining edge that is directed right if $\displaystyle\sum_{j_0=1}^n A^{(\ell)}_{i,j_0}=1$ and left if 
the sum is $0$. By $(1)$, $W$ is directed right if $\displaystyle\sum_{j_0=1}^{n-1} A^{(\ell)}_{i,j_0}=1$ and left if 
the sum is $0$. We know  $\displaystyle\sum_{j_0=1}^{n-1} A^{(\ell)}_{i,j_0}= \displaystyle\sum_{j_0=1}^{n} A^{(\ell)}_{i,j_0}$, so $W$ and $E$ are either both directed left or both directed right.

Similarly, if $1<i<n$, we know that  $\displaystyle\sum_{j_0=1}^n A^{(\ell-1)}_{j,j_0} + \displaystyle\sum_{j_0=1}^{n-i} A^{(\ell)}_{n+1-j_0,j} = \displaystyle\sum_{j_0=1}^n A^{(\ell-1)}_{j,j_0} + \displaystyle\sum_{j_0=1}^{n-i+1} A^{(\ell)}_{n+1-j_0,j}$, 
so by $(2)$, $N$ and $S$ either are both directed up or both directed down. 

If $i=1$ and $\ell$ is odd, then $N$ is a boundary edge directed up by $(4)$, and by $(2)$, $S$ is directed up as well.

If $i=1$ and $\ell$ is even, then $N$ is a boundary edge directed down by $(4)$, and by $(2)$, $S$ is directed down as well.

If $i=n$ and $\ell$ is odd, then $S$ is a chaining edge directed up if $\displaystyle\sum_{j_0=1}^n A^{(\ell-1)}_{i,j_0}=1$ and down if the sum is $0$. By $(2)$, $N$ is directed up if the $\displaystyle\sum_{j_0=1}^n A^{(\ell-1)}_{i,j_0}+A^{(\ell)}_{n,j}=1$ and down if the sum is $0$. Since $A^{(\ell)}_{n,j}=0$ by assumption, $N$ and $S$ is either both directed up or both directed down.

If $i=n$ and $\ell$ is even, then $S$ is a chaining edge directed down if $\displaystyle\sum_{j_0=1}^n A^{(\ell-1)}_{i,j_0}=1$ and up if the sum is $0$. By $(2)$, $N$ is directed down if the $\displaystyle\sum_{j_0=1}^n A^{(\ell-1)}_{i,j_0}+A^{(\ell)}_{n,j}=1$ and up if the sum is $0$. Since $A^{(\ell)}_{n,j}=0$ by assumption, $N$ and $S$ are either both directed up or both directed down.

Thus, $v_{i,j}^{(\ell)}$ is of one of the first four configurations  in Figure~\ref{fig:6V}.

\vspace{1ex}
\textbf{Case $A_{i,j}^{(\ell)}=1$:} 

If $1<j<n$, then by Definition~\ref{def:chainedASM} it must be that
$\displaystyle\sum_{j_0=1}^{j-1} A^{(\ell)}_{i,j_0}=0$ and $\displaystyle\sum_{j_0=1}^{j} A^{(\ell)}_{i,j_0}=1$, so by $(1)$ $W$ and $E$ 
are directed in opposite directions. If $\ell$ is odd, $W$ is directed right and $E$ is directed left. If $\ell$ is even, $W$ is directed left and $E$ is directed right. 

If $j=1$ and $\ell$ is  odd, then $W$ is a boundary edge that is directed right by $(4)$, and by $(1)$, $E$ is directed to the left.

If $j=1$ and $\ell$ is  even, then $W$ is a boundary edge that is directed left by $(4)$, and by $(1)$, $E$ is directed right.

If $j=n$ and $\ell$ is  odd, then $E$ is a chaining edge that is directed left by $(3)$ since $\displaystyle\sum_{j_0=1}^n A^{(\ell)}_{i,j_0}=1$. 
By $(1)$, $W$ is directed right since $\displaystyle\sum_{j_0=1}^{n-1} A^{(\ell)}_{i,j_0}=0$.

If $j=n$ and $\ell$ is  even, then $E$ is a chaining edge that is directed right by $(3)$ since $\displaystyle\sum_{j_0=1}^n A^{(\ell)}_{i,j_0}=1$. By $(1)$, $W$ is directed  left since 
$\displaystyle\sum_{j_0=1}^{n-1} A^{(\ell)}_{i,j_0}=0$.

Similarly, if $1<i<n$, we know by Definition~\ref{def:chainedASM} that $\displaystyle\sum_{j_0=1}^n A^{(\ell-1)}_{j,j_0} + \displaystyle\sum_{j_0=1}^{n-i+1} A^{(\ell)}_{n+1-j_0,j}=1$ and $\displaystyle\sum_{j_0=1}^n A^{(\ell-1)}_{j,j_0} + \displaystyle\sum_{j_0=1}^{n-i} A^{(\ell)}_{n+1-j_0,j}=0$, so 
by $(2)$ $N$ and $S$ are directed in opposite directions. If $\ell$ is odd, $N$ is directed up and $S$ is directed down. If $\ell$ is even, $N$ is directed down and $S$ is directed up. 

If $i=1$ and $\ell$ is  odd, then $N$ is a boundary edge directed up by $(4)$, and by $(2)$, $S$ is directed down.

If $i=1$ and $\ell$ is  even, then $N$ is a boundary edge directed down by $(4)$, and by $(2)$, $S$ is directed up.

If $i=n$ and $\ell$ is  odd, then $S$ is a chaining edge directed down by $(3)$ since $\displaystyle\sum_{j_0=1}^n A^{(\ell-1)}_{i,j_0}=0$. By $(2)$, $N$ is directed up since $\displaystyle\sum_{j_0=1}^n A^{(\ell-1)}_{i,j_0}+A^{(\ell)}_{n,j}=1$.

If $i=n$ and $\ell$ is  even, then $S$ is a chaining edge directed up by $(3)$ since $\displaystyle\sum_{j_0=1}^n A^{(\ell-1)}_{i,j_0}=0$. By $(2)$, $N$ is directed down since the $\displaystyle\sum_{j_0=1}^n A^{(\ell-1)}_{i,j_0}+A^{(\ell)}_{n,j}=1$.

In summary, if $\ell$ is odd, $N$ is directed up, $S$ is directed down, $W$ is directed right, and $E$ is directed left, so $v_{i,j}^{(\ell)}$ is in Configuration V. If $\ell$ is even, $N$ is directed down, $S$ is directed up, $W$ is directed left, and $E$ is directed right, so $v_{i,j}^{(\ell)}$ is in Configuration VI.

\vspace{1ex}
\textbf{Case $A_{i,j}^{(\ell)}=-1$:} 

If $1<j<n$, then by Definition~\ref{def:chainedASM} it must be that
$\displaystyle\sum_{j_0=1}^j A^{(\ell)}_{i,j_0}=0$ and $\displaystyle\sum_{j_0=1}^{j-1} A^{(\ell)}_{i,j_0}=1$, so by $(1)$ $W$ and $E$
are directed in opposite directions. If $\ell$ is odd, $W$ is directed left and $E$ is directed right. If $\ell$ is even, $W$ is directed right and $E$ is directed left.

We cannot have $j=1$ in this case, since then the partial row sum would be negative, contradicting Property $(1)$ of Definition~\ref{def:chainedASM}.

If $j=n$ and $\ell$ is odd, then $E$ is a chaining edge that is directed right by $(3)$ since $\displaystyle\sum_{j_0=1}^n A^{(\ell)}_{i,j_0}=0$.
By $(1)$, $W$ is directed left since $\displaystyle\sum_{j_0=1}^{n-1} A^{(\ell)}_{i,j_0}=1$. 

If $j=n$ and $\ell$ is  even, then $E$ is a chaining edge that is directed left by $(3)$ since $\displaystyle\sum_{j_0=1}^n A^{(\ell)}_{i,j_0}=0$. By $(1)$, $W$ is directed  right since 
$\displaystyle\sum_{j_0=1}^{n-1} A^{(\ell)}_{i,j_0}=1$.

Similarly, if $1<i<n$ we also know that $\displaystyle\sum_{j_0=1}^n A^{(\ell-1)}_{j,j_0} + \displaystyle\sum_{j_0=1}^{n-i+1} A^{(\ell)}_{n+1-j_0,j}=0$ and $\displaystyle\sum_{j_0=1}^n A^{(\ell-1)}_{j,j_0} + \displaystyle\sum_{j_0=1}^{n-i} A^{(\ell)}_{n+1-j_0,j}=1$, so by $(2)$ $N$ and $S$ are directed in opposite directions. If $\ell$ is odd, $N$ is directed down and $S$ is directed up. If $\ell$ is even, $N$ is directed up and $S$ is directed down. 

We cannot have $i=1$ in this case, since then the total row/column sum $\displaystyle\sum_{j=1}^n A^{(\ell-1)}_{i,j} + \displaystyle\sum_{j=1}^n A^{(\ell)}_{n+1-j,i}$ would be $0$, contradicting the maximality of Property $(3)$ of Definition~\ref{def:chainedASM} (since by Lemma~\ref{prop:asmsum}, each total row/column sum must equal $1$ so that the sum of all the entries in all the matrices is $\frac{nk}{2}$).

If $i=n$ and $\ell$ is  odd, then $S$ is a chaining edge directed up by $(3)$ since $\displaystyle\sum_{j_0=1}^n A^{(\ell-1)}_{i,j_0}=1$. By $(2)$, $N$ is directed down since $\displaystyle\sum_{j_0=1}^n A^{(\ell-1)}_{i,j_0}+A^{(\ell)}_{n,j}=0$.

If $i=n$ and $\ell$ is even, then $S$ is a chaining edge directed down by $(3)$ since $\displaystyle\sum_{j_0=1}^n A^{(\ell-1)}_{i,j_0}=1$. By $(2)$, $N$ is directed up since the $\displaystyle\sum_{j_0=1}^n A^{(\ell-1)}_{i,j_0}+A^{(\ell)}_{n,j}=0$.

In summary, if $\ell$ is odd, $N$ is directed down, $S$ is directed up, $W$ is directed left, and $E$ is directed right, so $v_{i,j}^{(\ell)}$ is in Configuration VI. If $\ell$ is even, $N$ is directed up, $S$ is directed down, $W$ is directed right, and $E$ is directed left, so $v_{i,j}^{(\ell)}$ is in Configuration V.

Therefore, each interior vertex is in one of the six configurations of Figure~\ref{fig:6V}.

\vspace{1ex}
Conversely, suppose a directed graph $G$ with underlying graph $GG_{n,k}$ has chained domain wall boundary conditions and each interior vertex has two edges entering and two edges leaving. We wish to show that $G$ is a chained ice configuration of a chained alternating sign matrix in $ASM_{n,k}^{\circ}$. 

Construct a tuple of matrices $A=\left(A^{(1)},A^{(2)},\ldots,A^{(k)}\right)$ as follows. If $v_{i,j}^{(\ell)}$ is in any of Configurations I--IV,
then let $A_{i,j}^{(\ell)}=0$. If $\ell$ is odd and $v_{i,j}^{(\ell)}$ is in Configuration V or if $\ell$ is even and $v_{i,j}^{(\ell)}$ is in Configuration VI,
let $A_{i,j}^{(\ell)}=1$. If $\ell$ is odd and $v_{i,j}^{(\ell)}$ is in Configuration VI or $\ell$ is even and $v_{i,j}^{(\ell)}$ is in Configuration V,
let $A_{i,j}^{(\ell)}=-1$. 

First note that Configurations I--IV each have horizontal edges both directed left or both directed right and vertical edges both directed up or both directed down. So given the boundary conditions and the placement of the vertices in Configurations V and VI, we may reconstruct the entire graph.  
Therefore, even though Configurations I--IV all map to $0$ entries in $A$, the map described above is injective. 

We wish to show $A$ is in $ASM_{n,k}^{\circ}$, so we need to show Properties $(1)-(3)$ in Definition~\ref{def:chainedASM}. 

\vspace{1ex}
Property $(1)$ holds because of the following.

\textbf{Case $\ell$ odd:}  By the chained domain wall boundary conditions ($(4)$ of Definition~\ref{def:chained_ice}), the boundary edges between $v_{i,0}^{(\ell)}$ and 
$v_{i,1}^{(\ell)}$
are directed inward (right) for all $i$.  So  $v_{i,1}^{(\ell)}$ is in one of Configurations I, IV, or V. Thus $A_{i,1}^{(\ell)}$ is either $0$ or $1$. From left to right across a row of $A^{(\ell)}$, a $0$ of Configuration I or IV may only be followed by another I or IV or a V, so the first nonzero entry of row $i$ in $A^{(\ell)}$ is $1$.

Using reasoning as in the previous sentences, each row $(v_{i,1}^{(\ell)},v_{i,2}^{(\ell)},\ldots,v_{i,n}^{(\ell)})$ of $G$ 
looks like the following (starting with I/IV or V, repeating cyclically and ending at any point): some number of I and/or IV configurations, then a single V, followed by some number of II and/or III, then a single VI, some number of I/IV, and so on. At each point in this sequence, the corresponding row partial sum $\displaystyle\sum_{j=1}^{m}A_{i,j}^{(\ell)}$ is $0$ or $1$, so Property $(1)$ is satisfied in this case. Now if the last entry $v_{i,n}^{(\ell)}$ in row $i$ is in one of Configurations II, III, or V, then row $i$ has a total sum of $1$; otherwise, row $i$ has a total sum of $0$. 

\textbf{Case $\ell$ even:}  By the chained domain wall boundary conditions, the boundary edges between $v_{i,0}^{(\ell)}$ and 
$v_{i,1}^{(\ell)}$ 
are directed outward (left) for all $i$. So  $v_{i,1}^{(\ell)}$ is in one of Configurations II, III, or VI. Thus $A_{i,1}^{(\ell)}$ is either $0$ or $1$. From left to right across a row of $A^{(\ell)}$, a $0$ of Configuration II or III may only be followed by another II or III or a VI, so the first nonzero entry of row $i$ in $A^{(\ell)}$ is $1$.

Using reasoning as in the previous sentences, each row 
$(v_{i,1}^{(\ell)},v_{i,2}^{(\ell)},\ldots,v_{i,n}^{(\ell)})$
looks like the following (starting with II/III or VI, repeating cyclically and ending at any point): some number of II and/or III configurations, then a single VI, followed by some number of I and/or IV, then a single V, some number of II/III, and so on. At each point in this sequence, the corresponding row partial sum $\displaystyle\sum_{j=1}^{m}A_{i,j}^{(\ell)}$ is $0$ or $1$, so Property $(1)$ is satisfied in this case. Now if the last entry in row $i$ is in one of Configurations I, IV, or VI, then row $i$ has a total sum of $1$; otherwise, row $i$ has a total sum of $0$. 

\vspace{1ex}
To show Property $(2)$, we examine the structure of the columns, together with their connecting row. 

\textbf{Case $\ell$ odd:} By the chained domain wall boundary conditions, the boundary edges between $v_{0,i}^{(\ell)}$ and $v_{1,i}^{(\ell)}$ are directed upward for all $i$. 
So $v_{1,i}^{(\ell)}$ is in one of Configurations I, III, or V. Using similar reasoning, each column $(v_{1,i}^{(\ell)},v_{2,i}^{(\ell)},\ldots,v_{n,i}^{(\ell)})$ looks like the following (starting with I/III or V, repeating cyclically and ending at any point): some number of I and/or III configurations, followed by a single V, then some number of II or IV, followed by a single VI, some number of I/III, and so on. At each point in this sequence, the corresponding column partial sum $\displaystyle\sum_{j=1}^{m}A_{j,i}^{(\ell)}$ is $0$ or $1$. 

Now to show Property $(2)$, recall that row $i$ of $A^{(\ell)}$ chains to  column $i$ of $A^{(\ell+1)}$. 

\textbf{Subcase $\displaystyle\sum_{j=1}^{n}A_{i,j}^{(\ell)}=0$}:
If the total sum of row $i$ of $A^{(\ell)}$ is $0$, then the chaining edge is directed from $v_{i,n}^{(\ell)}$ to $v_{n,i}^{(\ell+1)}$, so $v_{n,i}^{(\ell+1)}$ is in Configuration I, III, or VI. We follow the reverse cyclic rotation of configurations from what was described in the previous paragraph (since we are summing the columns of $A^{(\ell+1)}$ from bottom to top),
so the first nonzero entry in column i of $A^{(\ell+1)}$ from the bottom is $1$ (Configuration VI). By the analysis of the columns, the sum of row $i$ of $A^{(\ell)}$ added to the any partial sum from the bottom of column $i$ of $A^{(\ell+1)}$ is always $0$ or $1$. 

\textbf{Subcase $\displaystyle\sum_{j=1}^{n}A_{i,j}^{(\ell)}=1$}:
If the total sum of row $i$ of $A^{(\ell)}$ is $1$, then the chaining edge is directed from $v_{n,i}^{(\ell+1)}$ to $v_{i,n}^{(\ell)}$, so $v_{n,i}^{(\ell+1)}$ is in Configuration II, IV, or V. We again follow the reverse cyclic rotation of configurations, 
so the first nonzero entry in column i of $A^{(\ell+1)}$ from the bottom is $-1$ (Configuration VI). By the analysis of the columns, the sum of row $i$ of $A^{(\ell)}$ added to any partial sum from the bottom of column $i$ of $A^{(\ell+1)}$  is always $0$ or $1$. So Property $(2)$ holds in this case. Also, in either subcase, the total row/column sum $\displaystyle\sum_{j=1}^{n}\left(A_{i,j}^{(\ell)}+ A_{j,i}^{(\ell+1)}\right)$ is $1$.

\textbf{Case $\ell$ even:} By the chained domain wall boundary conditions, the boundary edges between $v_{0,i}^{(\ell)}$ and $v_{1,i}^{(\ell)}$ are directed downward for all $i$. 
So $v_{1,i}^{(\ell)}$ is in one of Configurations II, IV, or VI. Using similar reasoning,  each column $(v_{1,i}^{(\ell)},v_{2,i}^{(\ell)},\ldots,v_{n,i}^{(\ell)})$ looks like the following (starting with II/IV or VI, repeating cyclically and ending at any point): some number of II and/or IV, followed by a single VI, then some number of I and/or III, followed by a single V, then some number of II/IV, and so on. At each point in this sequence, the corresponding column partial sum $\displaystyle\sum_{j=1}^{m}A_{j,i}^{(\ell)}$ is $0$ or $1$. 

\textbf{Subcase $\displaystyle\sum_{j=1}^{n}A_{i,j}^{(\ell)}=0$}:
If the total sum of row $i$ of $A^{(\ell)}$ is $0$, then the chaining edge is directed from  $v_{n,i}^{(\ell+1)}$ to $v_{i,n}^{(\ell)}$, so $v_{n,i}^{(\ell+1)}$ is in Configuration II, IV, or V. We follow the reverse cyclic rotation of configurations from what was described in the previous paragraph (since we are summing the columns of $A^{(\ell+1)}$ from bottom to top),
so the first nonzero entry in column i of $A^{(\ell+1)}$ from the bottom is $1$ (Configuration V). By the analysis of the columns, the sum of row $i$ of $A^{(\ell)}$ added to the partial sum from the bottom of column $i$ of $A^{(\ell+1)}$ is always $0$ or $1$. 

\textbf{Subcase $\displaystyle\sum_{j=1}^{n}A_{i,j}^{(\ell)}=1$}:
If the total sum of row $i$ of $A^{(\ell)}$  is $1$, then the chaining edge is directed from $v_{i,n}^{(\ell)}$ to $v_{n,i}^{(\ell+1)}$, so $v_{n,i}^{(\ell+1)}$ is in Configuration I, III, or VI. We again follow the reverse cyclic rotation of configurations, 
so the first nonzero entry in column i of $A^{(\ell+1)}$ from the bottom is $-1$ (Configuration VI). By the analysis of the columns, the sum of row $i$ of $A^{(\ell)}$ added to the partial sum from the bottom of column $i$ of $A^{(\ell+1)}$ is always $0$ or $1$. So Property $(2)$ holds in this case.   Also, in either subcase, the total row/column sum $\displaystyle\sum_{j=1}^{n}\left(A_{i,j}^{(\ell)}+ A_{j,i}^{(\ell+1)}\right)$ is $1$.

\vspace{1ex}
To show Property $(3)$, recall that the maximum sum of entries in $ASM^{\circ}_{n,k}$ for $k$ even is $\frac{nk}{2}$. We have shown the total row/column sum for each connecting row/column pair is $1$. So the sum of all the entries is $\frac{nk}{2}$, proving Property $(3)$.

Thus $A$ is in $ASM_{n,k}^{\circ}$, and the map described in Definition~\ref{def:chained_ice} gives $G$. Thus $G$ is the chained ice configuration corresponding to the chained alternating sign matrix $A$.
\end{proof}

\section*{Acknowledgments}
Striker is supported in part by the National Security Agency grant  H98230-15-1-0041. The authors  thank the developers of Sage~\cite{sage} open source mathematical software, which was helpful for some calculations, as well as the anonymous referees for helpful comments. Most of the figures in this paper were drawn using Ipe~\cite{ipe}.

\bibliographystyle{plain}
\bibliography{rooks}

\end{document}